\documentclass[reqno,  oneside, 11pt]{amsart}
\usepackage{amsmath, mathrsfs, url, color,amssymb}
\usepackage[margin=2cm, a4paper]{geometry}
\usepackage{enumitem}
\usepackage{bbm}
\usepackage{graphicx}
\usepackage{subcaption}
\allowdisplaybreaks
\usepackage{color,longtable,graphicx,epstopdf}
\newtheorem{theorem}{Theorem}[section]
\newtheorem{corollary}{Corollary}[section]
\newtheorem{lemma}{Lemma}[section]
\newtheorem{proposition}{Proposition}[section]
\newtheorem{remark}{Remark}[section]

\newtheorem{asump}{Assumption}[section]
\allowdisplaybreaks

\definecolor{darkgreen}{rgb}{0,.6,0}

\title[An Explicit Euler-type Scheme for L\'evy-driven SDEs]{An Explicit Euler-type Scheme for L\'evy-driven SDEs with Superlinear and Time-Irregular Coefficients}

\usepackage[foot]{amsaddr}

\author{Sani Biswas$^{\mbox{\textasteriskcentered}}$}
\address{\textasteriskcentered$_{\normalfont{\mbox{Corresponding Author}}}$}
\address{\textasteriskcentered$_{\normalfont{\mbox{Centro de Modelamiento Matem\'atico, Universidad de Chile \& IRL 2807 - CNRS. E-Mail: {sbiswas@cmm.uchile.cl}}}}$}

\author{Joaqu\'in Fontbona$^{\mbox{\textdagger}}$}
\address{\textdagger$_{\normalfont{\mbox{Centro de Modelamiento Matem\'atico, Universidad de Chile \& IRL 2807 - CNRS. E-Mail: {fontbona@dim.uchile.cl}}}}$}

\begin{document}
	
	\begin{abstract}
		This paper introduces a {randomized tamed Euler scheme} tailored for L\'evy-driven stochastic differential equations (SDEs) with {superlinear random coefficients} and {Carath\'eodory-type drift}. Under assumptions that allow for time-irregular drifts while ensuring appropriate time-regularity of the diffusion and jump coefficients, the proposed scheme is shown to achieve the optimal strong $\mathcal{L}^2$-convergence rate, arbitrarily close to $0.5$.
		A crucial component of our methodology is the incorporation of drift randomization, which overcomes challenges due to low time-regularity, along with a taming technique to handle the superlinear state dependence.
		{Our analysis}  moreover covers  settings where the coefficients are random,  providing for instance  strong convergence of randomized tamed Euler schemes for {L\'evy-driven stochastic delay differential equations (SDDEs) with Markovian switching. To  our knowledge, this is the first {work} that addresses the case of superlinear coefficients in the numerical analysis of Carath\'eodory-type SDEs and even for ordinary differential equations}.
	\end{abstract}
	
	\maketitle
	\noindent
	\textbf{Keywords.} L\'evy-driven SDEs,   superlinear coefficients,  time-irregular drift, randomized tamed Euler scheme.
	\\ \\
	\textbf{AMS Subject Classifications.}  65C05, 65C30, 65C35, 60H35.
	\section{Introduction}
		Let $(\tilde\Omega, \tilde{\mathcal{F}}, \tilde{P})$ be a probability space equipped with a filtration $(\tilde{\mathcal{F}}_t)_{t \in [0,T]}$ satisfying the usual conditions.
	A standard $\mathbb{R}^m$-valued Brownian motion $w := (w_t)_{t \in [0,T]}$ is defined on this space, along with an independent Poisson random measure $n_p(dt,dz)$ {with intensity} $\rho(dz) \, dt$, where $\rho(\cdot)$ is specified on a sigma-finite measure space $(Z, \mathcal{Z})$ without the restriction $\rho(Z) < \infty$.
	Let
	\[
	b : \tilde\Omega\times [0,T]\times \mathbb{R}^d \to \mathbb{R}^d, \,\,
	\sigma :\tilde\Omega\times  [0,T] \times \mathbb{R}^d  \to \mathbb{R}^{d \times m} 
	\, \text{ and } \,
	\gamma : \tilde\Omega\times  [0,T] \times \mathbb{R}^d  \times Z \to \mathbb{R}^d
	\]
	be coefficients that are measurable functions with respect to the sigma-fields
	$
	\mathcal{P} \otimes \mathcal{B}(\mathbb{R}^d) 
	$ {and} 
	$\mathcal{P} \otimes \mathcal{B}(\mathbb{R}^d) \otimes \mathcal{Z},
	$
	respectively, for a fixed \( T > 0 \),
	where \( \mathcal{P} \) denotes the predictable \(\sigma\)-algebra on \( \tilde\Omega \times [0,T] \). We are interested in the numerical simulation of a solution \( (x_t)_{t\in[0,T]} \) to the  stochastic differential equation (SDE) 
	\begin{align} \label{eq:sde}
		x_{t} =  x_{0} 
		+ \int_{0}^{t} \mu(s,x_s) \, ds 
		+ \int_{0}^{t} \sigma(s,x_s) \, dw_s 
		+ \int_{0}^{t} \int_Z \gamma(s,x_s, z) \, \tilde{n}_p(ds, dz)
	\end{align}
	almost surely for any $t \in [0,T]$, {where \( \tilde{n}_p(ds,dz) := n_p(ds,dz) - \rho(dz)ds \) is the compensated Poisson random measure, the initial condition \( x_{0} \) is  \( \tilde{\mathcal{F}}_{0} \)-measurable and independent of \( w \) and \( n_p \) and 
		the drift coefficient $\mu:=\mu(t,x)$ is of Carath\'eodory type, that is, it is continuous with respect to $x$ and measurable with respect to $t$. Moreover, we  assume that the SDE coefficients have super-linear growth profiles.} {(Notice that, here and in the sequel, the coefficients' dependence on  randomness is not explicitly written).}

	Across a broad spectrum of disciplines---including finance, economics, engineering, medical sciences, and ecology---there is a growing reliance on models that account for sudden-event risks.
	In finance, for instance, asset prices frequently undergo sudden shifts triggered by sharp market collapses, monetary policy announcements, shifts in credit ratings, debt repayment failures, abrupt liquidity shortages, or high-frequency trading effects.
	To effectively capture these abrupt and irregular phenomena, stochastic differential equations (SDEs) driven by L\'evy noise have become an increasingly popular modeling choice, offering a more realistic framework than classical Brownian-based models (see \cite{Cont2004,Oksendal2007,Situ2006} and references therein). Yet,  {SDEs} rarely admit explicit solutions, which necessitates the development of robust numerical schemes to approximate their dynamics. In recent years, considerable research effort has focused on constructing both explicit and implicit numerical schemes for these equations. These studies have established rigorous results concerning their convergence in the strong and weak sense. For a detailed exploration of various numerical schemes and their orders of convergence, see \cite{Dereich2011,Higham2006,  Platen2010} and the literature cited therein.

	\subsection*{Results Available in the Literature}	
	Comprehensive discussions of both explicit and implicit numerical schemes for SDEs driven by L\'evy noise can be found in \cite{Platen2010} and the works cited therein. It is well established that the classical Euler scheme may suffer from finite-time moment blow-up when applied to SDEs with superlinearly growing coefficients—a phenomenon first rigorously demonstrated in \cite{hutzenthaler2010} for SDEs with continuous trajectories. While implicit methods can offer numerical stability in such regimes, they tend to be computationally expensive, particularly in high-dimensional settings, limiting their practical applicability.
	Motivated by these limitations, recent research has increasingly focused on designing explicit numerical schemes that remain both efficient and robust in the presence of superlinear coefficient growth. {Significant progress has been made in this direction, particularly through the development of tamed schemes, which limit the growth of the coefficients to ensure stability and maintain computational efficiency.} For SDEs with continuous paths, notable contributions include \cite{hutzenthaler2015, hutzenthaler2020, hutzenthaler2012, Kumar2020, Sabanis2013, Sabanis2016, Tretyakov2013}, among others. In the context of L\'evy-driven SDEs, relevant developments can be found in \cite{Chen2019, Dareiotis2016, Kumar2021a, Kumar2017a, Kumar2017}. 
	For taming schemes applied to L\'evy-driven McKean--Vlasov type SDEs, which lie beyond the scope of this work, interested readers may consult \cite{Biswas2024, Neelima2020} and references therein.

	{On the other hand}, SDEs with {Carath\'eodory}-type drift typically do not achieve convergence under the classical Euler scheme 
	due to time irregularity, as demonstrated in \cite[Section~3]{Przybylowicz2014}. 
	This limitation persists even in certain {simple case} scenarios, exemplified by integrals of the form 
	\(\int_0^t \mu(s,0) \, ds\) (cf.\ the example in \cite[p.~105]{Wozniakowski2006}). 
	To address these difficulties, Euler-integral algorithms have been developed that incorporate classical integrals 
	of the drift coefficient into the discretization.
	Further, {this can be modified} to establish Monte Carlo algorithms 
	that lead to randomized variants of the Euler-integral approach.   These rely solely on function evaluations {(and not on their integrals)}
	and retain the convergence rate established for the original Euler-integral scheme. 
	A detailed comparative study of the classical Euler method, the Euler-integral algorithm, and these randomized approaches, 
	together with rigorous error analysis and optimality results, is provided in \cite{Przybylowicz2014}.
	Randomized algorithms for SDEs have also been studied extensively in \cite{Przybylowicz2015a,Przybylowicz2015b} for the jump-free case, and in \cite{Przybylowicz2022} for the jump case, all focusing on the randomized Euler--Maruyama scheme along with comprehensive error bounds and optimality guarantees.
	Further results on error analysis and optimality for the randomized Milstein scheme in the absence of jumps 
	can be found in \cite{Kruse2019, Pawel2021}, where \cite{Kruse2019} additionally introduces 
	a two-stage variant and examines its error behavior. In settings involving jumps, relevant developments 
	appear in \cite{Verena2024}. Moreover, randomized Milstein schemes for McKean--Vlasov SDEs with common noise 
	are explored in \cite{Biswas2022}.

    	\subsection*{Difficulties, Novel Aspects and Summary of Contributions}
	One of the main difficulties in our analysis stems from the structure of the coefficients, which exhibit superlinear growth--a well-known obstacle for time-discrete schemes applied to SDE systems with such nonlinearities. To address cases involving superlinear coefficients, which require taming of the classical Euler scheme, a general approximation framework based on sequences {of the scheme's coefficients} will be introduced in \ref{sec:randomized_scm}. {A suggested formulation of the taming mechanism is also provided in \ref{sec:example} to aid intuition.} Related methodologies have been explored in \cite{Chen2019} for tamed Euler schemes applied to time-homogeneous L\'evy-driven SDEs with superlinear coefficients across all terms, and in \cite{Neelima2020} for Euler-type methods addressing more general time-inhomogeneous SDEs, including McKean--Vlasov-type L\'evy processes.  
	However, these works are not directly applicable in our setting due to the insufficient time regularity of the drift of the SDE \eqref{eq:sde}. In particular, it is known that SDEs with time-irregular drift may fail to achieve the desired convergence rates under classical Euler discretizations.

	To address this issue,    {tamed Euler-integral methods as developed in \cite{Dareiotis2016, Kumar2017a}} could be extended to the present setting, which involves computing the classical integral of the tamed drift $(\widehat{\mu})_{\tau m}^{n}(s,x)$, {defined in \ref{sec:randomized_scm}}.
	{Instead, we modify the Euler-integral methods to design a Monte Carlo algorithm, which yields a randomized variant (see Equation \eqref{eq:scm}) of the tamed Euler-integral scheme. This variant relies solely on pointwise evaluations of the drift while preserving the same convergence rate.}
	Our approach draws inspiration from recent advances on randomization techniques for Carath\'eodory-type SDEs with linearly growing coefficients--see, for example, \cite{Przybylowicz2015a,Przybylowicz2015b,Przybylowicz2014,Przybylowicz2022} together with references therein--and extends those ideas to the more challenging case of superlinearly growing coefficients.
	
	However, combining such randomization with drift taming introduces significant analytical challenges, particularly in establishing moment bounds (see Lemma~\ref{lem:scm_mb}) and in deriving an auxiliary estimate related to the strong convergence rate (which we do in Lemma~\ref{lem:estimate_last_term_MR}) of the numerical scheme \eqref{eq:scm}. This setting appears to lie beyond the scope of existing results in the literature. We are able to address these difficulties by imposing slightly stronger coercivity and monotonicity conditions than those commonly assumed. These conditions are required to accommodate scenarios in which the drift coefficient is evaluated at one time point, while the diffusion and jump coefficients are evaluated at another within the same summation (see Assumptions~\ref{asum:coercivity_p}, \ref{asum:coercivity_scm}, and \ref{asum:monotonocity_q0}). An illustrative example satisfying these conditions is provided by Equation~\eqref{eq:DWD}.
	%Alongside this, new technical developments are required to effectively manage the challenges in moment bound derivation posed by the interplay of randomization and taming.

	This work establishes in Theorem \ref{thm:main_result} strong $\mathcal{L}^{p}$-convergence of the randomized tamed Euler scheme \eqref{eq:scm}, with the optimal strong $\mathcal{L}^2$-convergence rate being arbitrarily close to $0.5$, in agreement with existing results in the literature. The optimality is obtained under appropriate time regularity assumptions on the diffusion and jump components, 
	specifically $\eta$-H\"older continuity with $\eta \in [1/2,1]$.
	The principal difficulty in establishing the convergence rate, as highlighted above, stems from the time-irregularity of the superlinear drift in the SDE \eqref{eq:sde}, which renders the combination of taming and randomization techniques particularly delicate and necessitates novel analytical approaches. Accordingly, the present analysis develops new estimates and introduces an auxiliary system governed by \eqref{eq:auxiliary_equation}, which serves as an intermediary between the original process \eqref{eq:sde} and its numerical approximation \eqref{eq:scm}; see Lemmas~\ref{lem:estimate_first_term_MR} and \ref{lem:estimate_last_term_MR} below.
	To the best of our knowledge, this is the first {work} that addresses the case of superlinear coefficients in the numerical analysis of Carath\'eodory-type SDEs, even in the context of ordinary differential equations.
	Furthermore, the analytical techniques introduced here may serve as a foundation for future studies on the convergence behaviour of higher-order explicit numerical schemes for Carath\'eodory-type SDEs driven by L\'evy noise with superlinear coefficients.
	
	Moreover, the general random structure of our coefficient conditions allows the proposed class of taming schemes to be adapted to a broad range of SDEs, including stochastic delay differential equations (SDDEs) and SDEs with Markovian switching, both under L\'evy-driven dynamics (see \ref{sec:application}). 
	Our framework also naturally encompasses models from { applied sciences. For instance,  let us consider the  stochastic FitzHugh--Nagumo model widely studied in the mathematical neuroscience literature. Motivated by the classic FitzHugh--Nagumo model} \cite{FitzHugh1969, Nagumo1962}, the works \cite{Baladron2012, Tuckwell2003} among many others, studied the following  system of SDE:
	\begin{align*}
		d v_t &= \bigl( v_t - \tfrac{v_t^3}{3} - r_t + I_{\mathrm{ext}}(t) \bigr)\, dt 
		+ \sigma_v\, d w^v_t,\\[1.2ex]
		d r_t &= \epsilon \bigl( v_t + a - b r_t \bigr)\, dt
	\end{align*}
	{where $v_t$ denotes the membrane potential, $r_t$ is a recovery variable,  $w^v_t$ is a standard Wiener process and  $\epsilon, a, b,    \sigma_v > 0$ are model parameters (with $ \sigma_v$ the noise intensity). Here, $I_{\mathrm{ext}}(t)$ is a time-inhomogeneous external input current which,  in practical situations,  could  be  quite irregular. Our general results  apply in such setting, and cover moreover case of  L\'evy noises more general than Brownina motion. The Markov switching framework discussed above further allows us to consider input random currents  of the  form $I_{\mathrm{ext}}(t)=I^*_{\mathrm{ext}}(t,\alpha_t)$, i.e. signals which randomly alternate between several possible biologically meaningful regimes or patterns.} 
	
	To validate the theoretical observations, numerical simulations are furthermore performed in \ref{sec:numerics} on a new variant of well-known equations that fit the framework perfectly, namely for  stochastic double-well dynamics.
	%Ginzburg--Landau-type equations 
	%	(see \cite{Kumar2020,Kumar2021}) 
	%and stochastic double-well dynamics. 
	%(see \cite{Beyn2017, Reis2022}). 
	%For these examples, we also provide a comparative study of our algorithm and the Euler-integral algorithm from a numerical perspective.
    
		We conclude this section with some useful notations.
	
	\subsection*{Notations} 
	The Euclidean norm on \(\mathbb{R}^d\) and the Hilbert--Schmidt  norm on \(\mathbb{R}^{d \times m}\) are both denoted by \(|\cdot|\). Given a topological space \(A\), its Borel sigma-algebra is denoted by \(\mathcal{B}(A)\). 
	The indicator function of a set \(B\) is written as \(\mathbf{1}_B\).
	The expectation of a random variable \(X\) is denoted by \(E(X)\), and the conditional expectation given the sigma-algebra \(\mathcal{F}_t\) by \(E_t(X)\). The notation \(X \in \mathcal{L}^p\) means that \(X\) has finite \(p\)-th moment, i.e., \(E\left(|X|^p\right) < \infty\). The set \(\mathcal{L}_{0+}^p\) consists of all nonnegative, \(\tilde{\mathcal{F}}_0\)-measurable random variables with finite \(p\)-th moment. A generic positive constant independent of the discretization step-size is denoted by \(K\), whose value may change line to line.
	
	\section{Required Preliminaries and Elaboration of Key Findings}  %\label{sec:main_result}
This section first brings up the well-posedness and moment stability of the  SDE \eqref{eq:sde}.
	Moreover, Theorem \ref{thm:main_result} provides the main result of this article on the strong convergence rate of the proposed randomized tamed Euler scheme \eqref{eq:scm} for the SDE \eqref{eq:sde}, with its proof in \ref{sec:convergence_rate}. 
	
	\subsection{Well-posedness and Moment Bounds}
	Let us make the following assumptions to ensure the well-posedness and moment stability of the SDE \eqref{eq:sde}. 
	First,  fix $q\geq 4$.

    \begin{asump}\label{asum:ic}
		$E|x_0|^{q}<\infty$.
	\end{asump}

    \begin{asump} \label{asum:monotonocity}
		There exists a   constant $C>0$,  such that 
		\begin{align*}
			2(x-y)\big(\mu(s,x)-\mu(s,y)\big)+ &|\sigma(s,x)-\sigma(s,y)|^2 +\int_Z \big| \gamma(s,x,z)-\gamma(s,y,z)\big|^2\rho(dz) 
			\\
            &\leq  C|x-y|^2
		\end{align*}
		for any $s\in[0,T]$ and $x, y\in \mathbb R^d$.
	\end{asump}
	
	\begin{asump}
		\label{asum:coercivity}
		There exist a constant \( C > 0 \) and  a  random variable \( \Upsilon\in\mathcal{L}_{0+}^{2} \), such that
		\begin{align*}
			2x\mu(s,x)+|\sigma(s,x)|^2 +\int_Z\big| \gamma(s,x,z)\big|^2\rho(dz) \leq  C(\Upsilon+|x|)^2
		\end{align*}
		for any \( s \in [0,T] \) and \( x \in \mathbb{R}^d \).
		
	\end{asump}
	
	\begin{asump} \label{asum:continuity}
		The map $\mathbb R^d\ni x\mapsto \mu(s,x)$ is  continuous for any $s\in  [0,T]$.
	\end{asump}

	\begin{asump} \label{asum:coercivity_p}
		There exist a constant \( C > 0 \) and a  random variable \( \Upsilon\in\mathcal{L}_{0+}^{q} \), such that
		\begin{align*}
			&2|x|^{{q}-2}x\mu(s,x)
			+(q-1)|x|^{{q}-2}\big|\sigma(t,x)\big|^2  \notag
			\\
			&\quad 
			+2(q-1)\int_Z\big| \gamma(t,x,z) \big|^2\int_{0}^1 (1-\theta)  \big |x+\theta|\gamma(t,x,z)| \big|^{q-2} d\theta \rho(dz) \leq C(\Upsilon+|x|)^{q}
		\end{align*}
		for any \( s,t \in [0,T] \) and \( x \in \mathbb{R}^d \).
	\end{asump}

{\begin{remark}
			The conditions imposed on our coefficients allow for random (predictable) dependence in the time variable, which leads to the appearance of a random variable~$\Upsilon$ on the right-hand side in some of the previous (and forthcoming) assumptions. 
			Such a framework has been considered, for instance, in \cite{Dareiotis2016,Kumar2017a}. The term ``Carath\'eodory-type '' drift refers to the partial continuity  stated in \ref{asum:continuity}
	\end{remark}}

    The following proposition establishes the well-posedness and moment stability of the SDE \eqref{eq:sde}, specialized from the more general result in \cite[Theorem 2.1]{Neelima2020}.
	\begin{proposition} \label{prop:mb:sde}
		If Assumptions \mbox{\normalfont  \ref{asum:monotonocity}} to  \mbox{\normalfont \ref{asum:continuity}}  are met, the SDE \eqref{eq:sde} has a unique solution. . 
		Further, if Assumptions \mbox{\normalfont \ref{asum:ic}} and \mbox{\normalfont \ref{asum:coercivity_p}} are also true, there exists a constant $K>0$, such that
		$
		\displaystyle\sup_{t\in[0,T]}E|x_t|^{q}\leq K.
		$
	\end{proposition}
	
	\subsection{Randomized Tamed Euler Scheme and Main Result}  \label{sec:randomized_scm}
	It is well known that the classical Euler scheme fails to converge in the strong $\mathcal{L}^p$-sense for SDEs with superlinear coefficients, necessitating taming strategies (see~\cite{Chen2019, Neelima2020}). {To construct a tamed Euler scheme for the SDE~\eqref{eq:sde}, we consider the sequences $\{(\widehat{\mu})_{\tau m}^{n}(s,x)\}_{n\in \mathbb{N}}$, $\{(\widehat{\sigma})_{\tau m}^{n}(s,x)\}_{n\in \mathbb{N}}$, and $\{(\widehat{\gamma})_{\tau m}^{n}(s,x,z)\}_{n\in \mathbb{N}}$.
		These are measurable functions with respect to $\mathcal{P} \otimes \mathcal{B}(\mathbb{R}^d)$ and $\mathcal{P} \otimes \mathcal{B}(\mathbb{R}^d) \otimes \mathcal{Z}$, respectively, and satisfy the conditions in Assumptions~\textnormal{\ref{asum:coercivity_scm}}, \textnormal{\ref{asum:tame}}, and~\textnormal{\ref{asum:convergence}} below.}  The subscript `$\tau m$' denotes the taming form of the sequences required to ensure stability and convergence for superlinear coefficients. In the linearly growing case, taming is unnecessary and the subscript `$\tau m$' may be omitted.

	Inspired by the randomization framework in~\cite{Przybylowicz2015a, Przybylowicz2015b, Przybylowicz2014, Przybylowicz2022}, the difficulties posed by the time-irregular superlinear drift in SDE~\eqref{eq:sde} are addressed by randomizing the tamed drift $(\widehat{\mu})_{\tau m}^{n}(s,x)$, leading to the randomized tamed Euler scheme to be introduced in equation \eqref{eq:scm} below. 
	{To define the scheme, let $\varphi = (\varphi_{k-1})_{k \in \mathbb{N}}$ be an i.i.d.\ sequence of $\mathcal{U}(0,1]$-distributed random variables on a different probability space $(\Omega^{\varphi}, \mathcal{F}^{\varphi}, P^{\varphi})$, and $\mathbb{F}^{\varphi} = (\mathcal{F}_{t}^{\varphi})_{t \geq 0}$ be the filtration defined by $\mathcal F_0^\varphi:= \{\emptyset, \Omega^\varphi\}$, and  by $\mathcal{F}_{t}^{\varphi}=\sigma(\varphi_0,\ldots,\varphi_{k-1})$, for $t \in (t_{k-1},t_k]$ and all $ k\in\mathbb N$. Denote expectation with respect to $P^\varphi$ by ${E}^\varphi$.} 
	{Further, consider the product probability space
		\(
		(\Omega, \mathcal{F}, P) := (\tilde{\Omega} \times \Omega^\varphi, \tilde{\mathcal{F}} \otimes \mathcal{F}^\varphi, \tilde{P} \otimes P^\varphi),
		\) in which $\varphi, x_0$, $w$, and $n_p$ are independent, and 
		equip it  with the continuous-time product filtration
		\(
		\big(\mathcal F_t :=\tilde{\mathcal F}_t \otimes \mathcal F^\varphi_{t}\big)_{t\in[0,T]}.
		\)
		%where for $t \in (t_{k-1},t_k]$, $\mathcal F^\varphi_{\kappa_n(t)} := \mathcal F^\varphi_{k-1}$, so that the $\varphi$-information is piecewise constant on each interval.
		The filtration $(\mathcal F_t)_{t\in[0, T]}$ is augmented to satisfy the usual conditions (complete and right-continuous).
		% Denote expectation with respect to $P$ by ${E}$.
	}
	Let $0 = t_0 < t_1 < \cdots < t_n = T$ be an equidistant partition of $[0,T]$ with $\Delta t = t_k - t_{k-1}$. For $t \in [t_{k-1}, t_k)$, define the left-endpoint and randomized evaluation points by
	\(
	\kappa_n(t) := t_{k-1}, \, \xi_n(t) := t_{k-1} + \Delta t \varphi_{k-1}.
	\)

	Now,  we  propose a  randomized tamed Euler  scheme for the Carath\'eodory-type SDE \eqref{eq:sde}, defined by
	\begin{align}  \label{eq:scm}
		x_t^{n}= x_{0}+\int_{0}^{t}  &(\widehat\mu)_{\tau m}^{\displaystyle{n}}\big(\xi_n(s),x_{\kappa_n(s)}^{n}\big) ds +\int_{0}^{t}   (\widehat\sigma)_{\tau m}^{\displaystyle{n}}(\kappa_n(s),x_{\kappa_n(s)}^{n}) dw_s    \notag
		\\
		&+\int_{0}^{t}\int_Z  (\widehat\gamma)_{\tau m}^{\displaystyle{n}}\big(\kappa_n(s),x_{\kappa_n(s)}^{n}, z\big)   \tilde n_p(ds,dz) 
	\end{align}
	almost surely for all $t\in[ 0, T]$.
	
	\begin{remark}
		\noindent
		%The general taming formulation encompasses cases with Lipschitz continuous coefficients, i.e., linear growth in the state variable.
		Tamed Euler schemes for L\'evy-driven SDEs developed in~\cite{Dareiotis2016, Kumar2017a} arise as special cases of the randomized tamed Euler framework~\eqref{eq:scm}, assuming a linear jump coefficient and interpreting the scheme as an Euler-integral method with time-integrated taming.
		Furthermore, for L\'evy-driven SDEs with linearly growing coefficients, scheme~\eqref{eq:scm} reduces to a randomized Euler scheme without taming, as studied in~\cite{Przybylowicz2022}; the case $\gamma \equiv 0$ corresponds to~\cite{Przybylowicz2014}.
	\end{remark}

	The convergence rate of the scheme \eqref{eq:scm} for the SDE \eqref{eq:sde} will be established in Theorem~\ref{thm:main_result} below, under the following set of assumptions. First, fix ${\zeta}>0$.
	\begin{asump} \label{asum:coercivity_scm}
		%For some $q\geq 3$, 
		There exist a constant \( C > 0 \) independent of $n\in \mathbb N$  and a sequence of  random variables  \(( \Upsilon_n)_{n\in\mathbb N} \) in \(\mathcal{L}_{0+}^{q}\)  with $\sup_{n\in\mathbb N}E(\Upsilon_n^q)<\infty$, such that
		\begin{align*}
			&2|x|^{q-2}x(\widehat{\mu})_{\tau m}^{\displaystyle{n}}(s,x) +(q-1)|x|^{q-2}\big| (\widehat{\sigma})_{\tau m}^{\displaystyle{n}}(t,x) \big|^2	+2(q-1)\int_Z\big| (\widehat{\gamma})_{\tau m}^{\displaystyle{n}}(t,x,z) \big|^2 \notag
			\\
			&\qquad \qquad 
			\times\int_{0}^1 (1-\theta)  \big |x+\theta|(\widehat{\gamma})_{\tau m}^{\displaystyle{n}}(t,x,z)| \big|^{q-2} d\theta \rho(dz) \leq C\big(\Upsilon_{n}+|x|)^{q}
		\end{align*}
		for any $s,t\in[0,T]$ and $x\in\mathbb R^d$, 
	\end{asump}

	\begin{asump} \label{asum:tame}
		There exist
		a constant $C>0$ independent of $n\in\mathbb N$ and a sequences of  random variables  \(( \Upsilon_n)_{n\in\mathbb N} \) in \(\mathcal{L}_{0+}^{q}\) with $\sup_{n\in\mathbb N}E(\Upsilon_n^q)<\infty$, such that 
		\begin{align*}
			|(\widehat{\mu})_{\tau m}^{\displaystyle{n}}(s,x)|&\leq   C \min\big\{n^{\frac{1}{3}}\big(\Upsilon_{n}+|x|\big),|\mu(s,x)|\big\},
			\\
			|(\widehat\sigma)_{\tau m}^{\displaystyle{n}}(s,x)|&\leq C\min\big\{n^{\frac{1}{6}}\big(\Upsilon_{n}+|x|\big),|\sigma(s,x)|\big\},
			\\
			\int_Z|(\widehat\gamma)_{\tau m}^{\displaystyle{n}}(s,x,z)|^{p_0}\rho(dz)&\leq C\min\big\{n^{\frac{1}{3}}(\Upsilon_{n}+|x|)^{p_0},\int_Z|\gamma(s,x,z)|^{p_0}\rho(dz)\big\} 
		\end{align*}
		for any  $s\in[0,T]$, $x\in\mathbb R^d$ and $2\leq p_0\leq q$.
	\end{asump}

	%\begin{asump} \label{asum:superlin_jump}
	%There exists   a  constant $C>0$ such that for any  $x\in\mathbb R^d$ and $\mu\in\mathcal{P}_2(\mathbb R^d)$,
	%\begin{align*}
	%&\int_Z |\gamma(s,x,z) |^{p_0} \rho(dz)   \leq K \{(1+|x|)^{\zeta/2+p_0}\} \mbox{ for some  $2< p_0\leq q$. } 
	%\end{align*}
	%\end{asump}
	
	\begin{asump} \label{asum:at_0}
		There exist a constant \( C > 0 \) and a  random variable  \( \Upsilon\in\mathcal{L}_{0+}^{q} \), such that
		\begin{align*}
			|\mu(s,0)|+|\sigma(s,0)|+\int_Z|\gamma(s,0,z)|^{p_0}\rho(dz)
			&\leq \Upsilon
		\end{align*}
		for any \( s \in [0,T] \), \( x \in \mathbb{R}^d \) and $2\leq p_0\leq q$.
	\end{asump}

	\begin{asump} \label{asum:poly_lip_drift}
		There is a constant $C>0$,  such that 
		\begin{align*}
			|\mu(s,x)-\mu(s,y)|
			&\leq C \left(1+|x|+|y|\right)^{\zeta}|x-y|
		\end{align*}
		for any $s\in[0, T]$ and $x, y\in\mathbb R^d$.
	\end{asump}

	\begin{asump} \label{asum:poly_lip_jump}
		There is a constant $C>0$,  such that 
		\begin{align*}
			&\int_Z |\gamma(s,x, z)-\gamma(s,y, z)|^{p_0} \rho(dz)   \leq  C \left(1+|x|+|y|\right)^{\zeta}|x-y|^{p_0}
		\end{align*}
		for any $s\in[0, T]$, $x, y\in\mathbb R^d$ and $2< p_0\leq q$.
	\end{asump}
	
	\begin{asump}  \label{asum:holder_time_diffusion_jump}
		There exist    constants $C>0$, $\alpha,\beta\in (0,1]$ and  a  random variable \( \Upsilon\in\mathcal{L}_{0+}^{ q} \),   such that
		\begin{align*}
			|{\sigma}(s,x)-{\sigma}(t,x)|&\leq C|s-t|^{\alpha} \big(\Upsilon+|x|^{\zeta/2+1}\big),
			\\
			\int_Z|{\gamma}(s,x,z)-{\gamma}(t,x, z)|^{p_0}\rho(dz)&\leq C|s-t|^{\beta p_0} \big(\Upsilon+|x|^{\zeta+p_0}\big) \mbox{ for any  $2\leq p_0\leq q$ }
		\end{align*}
		for any  $s,t\in[0,T]$ and  $x\in\mathbb R^d$.
	\end{asump} 
	
	\begin{asump} \label{asum:convergence}
		There 
		exist an $\delta\in(0,1)$ and 
		a constant $C>0$ independent of $n\in\mathbb N$,    such that 
		\begin{align*}
			&E\big|\mu\big(\xi_n(s),x_{\kappa_n(s)}^{n}\big) -(\widehat\mu)_{\tau m}^{\displaystyle{n}}\big(\xi_n(s),x_{\kappa_n(s)}^{n}\big)\big|^{p_0}
			\\
			&\qquad\qquad\qquad\quad\,\,+E\big|\sigma(\kappa_n(s),x_{\kappa_n(s)}^{n})- (\widehat\sigma)_{\tau m}^{\displaystyle{n}}(\kappa_n(s),x_{\kappa_n(s)}^{n})\big|^{p_0} \leq C n^{-\frac{p_0}{p_0+\delta}},
			\\
			&\Big(\int_Z E\big|\gamma(\kappa_n(s),x_{\kappa_n(s)}^{n},z)-(\widehat\gamma)_{\tau m}^{\displaystyle{n}}\big(\kappa_n(s),x_{\kappa_n(s)}^{n}, z\big)\big|^{\bar q}\rho(dz)\Big)^{p_0/\bar q}\leq C n^{-\frac{ p_0}{ p_0+\delta}}  
		\end{align*}
		for any  $s\in[0,T]$ and  {$2\leq p_0\leq 2q/(5\zeta+3)$}, and $\bar q=2, p_0$.
	\end{asump}
	
	\begin{asump} \label{asum:monotonocity_q0}
		There exist   constants $C>0$ and   $\lambda>1$,
		%%and a sequence of nonnegative $\mathcal F^{\varphi}$-adapted process $(\mathcal C_t^n)_{t\in[0,T]}$  with $\sup_{n\in \mathbb N}\sup_{s\in [0,T]}E^{\varphi} \mathcal (\mathcal C_s^n)<\infty$,  
		such that 
		\begin{align*}
			p_0&|x-y|^{p_0-2}(x-y)\big(\mu(s,x)-\mu(s,y)\big)
			+\tfrac{p_0(p_0-1)\lambda}{2}|x-y|^{{p_0}-2}\big|\sigma(t, x)-\sigma(t, y)\big|^2
			\\
			& +
			{p_0(p_0-1)\left(2^{p_0-4}\mathbf{1}_{\{p_0>3\}}+2^{p_0-4}\mathbf{1}_{\{ p_0=2\}}+\tfrac{1}{2}\mathbf{1}_{\{2<p_0\leq 3\}}\right)} \notag
			\\
			&\quad\times \int_Z\Big(\lambda|x-y|^{{p_0}-2}\big|\gamma(t, x,z) -\gamma(t, y,z)\big|^2 +\lambda^{p_0-1}\big|\gamma(t, x,z) -\gamma(t, y,z)\big|^{p_0}\Big) \rho(dz) 
            \\
        &\qquad\qquad\qquad\qquad\qquad\qquad\leq C|x-y|^{p_0}
		\end{align*}
		for any $s,t\in[0,T]$, $x, y\in \mathbb R^d$ and $2\leq p_0\leq q$.
	\end{asump}

	The primary contribution of this work is encapsulated in the following statement. A detailed proof of this result is provided in \ref{sec:convergence_rate}.

	\begin{theorem} \label{thm:main_result}
		Let $q_0\geq 2$, $\zeta>0$ and  $\delta\in(0,1)$ be given and suppose that Assumptions \mbox{\normalfont  \ref{asum:ic}}, \mbox{\normalfont  \ref{asum:monotonocity}} and \mbox{\normalfont  \ref{asum:coercivity_p}}
		to \mbox{\normalfont  \ref{asum:monotonocity_q0}} hold with $q\geq \max\big\{q_0\delta^{-1}(q_0+\delta)\zeta,(q_0+\delta)(\zeta+1), {q_0(5\zeta+3)/2}\big\}$.  Then, there is a constant $K>0$   independent of $n\in \mathbb N$, such that
		\begin{align*}
			\sup_{ t\in[0,T]}E|x_t-x_t^{n}|^{q_0}\leq K n^{-\min\big\{\frac{q_0}{q_0+\delta},\,\alpha q_0,\, \beta q_0\big \}}. 
		\end{align*}
		The latter holds in particular for $q_0=2$ and we deduce that  for any $\hat q\in(0,2]$ there is a constant $K>0$   independent of $n\in \mathbb N$, such that
		\[	\sup_{ t\in[0,T]}E|x_t-x_t^{n}|^{\hat q}\leq K n^{-\min\big\{\frac{\hat q}{2+\delta},\,\alpha \hat q,\, \beta \hat q\big \}}.\]
		
	\end{theorem}
		We readily deduce: 
	\begin{corollary}
		\label{cor:optimality_discuss}
		Under the assumptions of Theorem~\ref{thm:main_result}, suppose moreover that Assumption \ref{asum:holder_time_diffusion_jump} holds with time regularity 
		$\alpha, \beta \in [1/2,1]$. Then, the randomized tamed Euler method~\eqref{eq:scm} 
		attains the $\mathcal{L}^2$-convergence rate ${1}/{(2+\delta)}$. In particular,  this rate can be made arbitrarily close to the optimal rate $0.5$ if  $\delta\in(0,1)$ in Assumption  \ref{asum:convergence} can be taken arbitrarily small and $q\geq \max\big\{4, q_0\delta^{-1}(q_0+\delta)\zeta,(q_0+\delta)(\zeta+1), {q_0(5\zeta+3)/2}\big\}$ in the other assumptions can  be taken sufficiently large.
	\end{corollary}

	\begin{remark}
		 Evidence from numerical experiments presented in \ref{sec:numerics} suggests that the actual convergence error and, therefore, sharp upper bounds for it,  might depend on the value of $q_0$. These questions are left for future work.
	\end{remark}

	A concrete example of taming coefficients satisfying the conditions in Theorem~\ref{thm:main_result} and Corollary \ref{cor:optimality_discuss} is provided next.
	
	\subsection{{Suggested} Formulation of Taming}  \label{sec:example} 
	
	By introducing the following appropriately modified (tamed) coefficients into the randomized tamed Euler scheme~\eqref{eq:scm}, we obtain a particular form of this scheme for the SDE~\eqref{eq:sde}.
	\begin{equation}
		\begin{aligned}
			(\widehat{\mu})_{\tau m}^{\displaystyle{n}}(s,x) := &\frac{\mu(s,x)}{1 + n^{-1/2} |x|^{3\zeta/2}}, \,\,
			(\widehat\sigma)_{\tau m}^{\displaystyle{n}}(s,x) := \frac{\sigma(s,x)}{1 + n^{-1/2}  |x|^{3\zeta/2}}
			\\
			&(\widehat\gamma)_{\tau m}^{\displaystyle{n}}(s, x,z):=\frac{\gamma(s,x,z)}{1+n^{-1/2}|x|^{3\zeta/2}}  
		\end{aligned} \label{eq:tame_coeffs}
	\end{equation}
	for any \( s \in [0,T] \),  \( x \in \mathbb{R}^d \) and $z\in Z$. Verification of the taming-related assumptions--namely Assumptions~\ref{asum:coercivity_scm}, \ref{asum:tame}, and \ref{asum:convergence}--can be found in Appendix~\ref{sec:app}.

	In the next section, an extension of the taming scheme \eqref{eq:scm} to stochastic delay systems driven by both continuous and jump noise, and evolving under regime switching governed by a Markov chain, is presented. This extension is facilitated by the general random structure of the coefficient assumptions.
	\subsection{Randomized Tamed Euler Scheme for  L\'evy-driven  SDDE with Markovian Switching} \label{sec:application}
	Let \( m_0 \in \mathbb{N} \) be fixed and let \( \alpha := (\alpha_t)_{t \geq 0} \) be a continuous-time Markov chain taking values in the finite state space \( \mathbb{S} := \{1, 2, \ldots, m_0\} \), with generator matrix \( Q = (q_{i_0 j_0})_{i_0, j_0 \in \mathbb{S}} \). That is, for any \( \varepsilon > 0 \), the transition probabilities satisfy
	\[
	{P}^\alpha(\alpha_{t+\varepsilon} = j_0 \mid \alpha_t = i_0) = 
	\begin{cases}
		q_{i_0 j_0} \varepsilon + o(\varepsilon), & \text{if } i_0 \neq j_0, \\
		1 + q_{i_0 j_0} \varepsilon + o(\varepsilon), & \text{if } i_0 = j_0
	\end{cases}
	\]
	where \( o(\varepsilon) \) denotes the Bachmann--Landau little-\( o \) notation as \( \varepsilon \to 0 \), with \( q_{i_0 j_0} \geq 0 \) for \( i_0 \neq j_0 \), and \( q_{i_0 i_0} := -\sum_{j_0 \neq i_0} q_{i_0 j_0} \) for each \( i_0 \in \mathbb{S} \).
	Let \,
	$
	\tilde{\mu} : [0,T]\times \mathbb{R}^d\times \mathbb{R}^d\times \mathbb{S} \to \mathbb{R}^d, \,
	\tilde{\sigma} : [0,T]\times \mathbb{R}^d\times \mathbb{R}^d\times \mathbb{S} \to \mathbb{R}^{d\times m}
	$
	and \,
	$
	\tilde{\gamma} : [0,T]\times \mathbb{R}^d\times \mathbb{R}^d\times \mathbb{S} \times Z \to \mathbb{R}^d
	$
	be measurable functions, and let \( \theta > 0 \) be a fixed delay parameter. Now, for any \( t \in [0,T] \),  consider the following L\'evy-driven  SDDE with Markovian switching
	\begin{align} \label{eq:sdde}
		x_t = x_0 
		&+ \int_0^t \tilde{\mu}(s, x_s, x_{s-\theta}, \alpha_s) \, ds 
		+ \int_0^t \tilde{\sigma}(s, x_s, x_{s-\theta}, \alpha_s) \, dw_s \notag \\
		&+ \int_0^t \int_Z \tilde{\gamma}(s, x_s, x_{s-\theta}, \alpha_s, z) \, \tilde{n}_p(ds, dz)
	\end{align}
	almost surely 
	with initial data \( \zeta_t := x_{t-\theta} \) for any \( t \in [0, \theta] \).
	Then, Equations \eqref{eq:scm} and \eqref{eq:tame_coeffs} produces the randomized tamed Euler scheme for the L\'evy-driven SDDE \eqref{eq:sdde}  with Markovian switching as follows
	\begin{align} %\label{eq:euler_delay_switching}
		x_t^{n} =\;& x_0 
		+ \int_0^t \frac{\tilde{\mu}\big(\xi_n(s), x_{\kappa_n(s)}^{n}, x_{\kappa_n(s-\theta)}^{n}, \alpha_{\kappa_n(s)}\big)}{1 + n^{-1/2} |x|^{3\zeta/2}}\,ds \notag \\
		& + \int_0^t \frac{\tilde{\sigma}\big(\kappa_n(s), x_{\kappa_n(s)}^{n}, x_{\kappa_n(s-\theta)}^{n}, \alpha_{\kappa_n(s)}\big)}{1 + n^{-1/2} |x|^{3\zeta/2}}\,dw_s \notag \\
		& + \int_0^t \int_Z \frac{\tilde{\gamma}\big(\kappa_n(s), x_{\kappa_n(s)}^{n}, x_{\kappa_n(s-\theta)}^{n}, \alpha_{\kappa_n(s)}, z \big)}{1 + n^{-1/2} |x|^{3\zeta/2}}\,\tilde{n}_p(ds,dz) \notag
	\end{align}
	for any $t\in [0,T]$.	
	Based on the ideas developed in \cite{Dareiotis2016, Kumar2017a, Neelima2020}, one can establish strong convergence for the above randomized tamed Euler scheme, analogous to the result stated in Theorem~\ref{thm:main_result}.
	
	\subsection{Required Preliminaries}
	The following elementary result is needed,  a proof of which can be found in \cite[Appendix A]{Biswas2024}. 
	\begin{lemma} \label{lem:mvt}
		For any $a,b\in\mathbb R^d$ and  $p\geq 2$,
		\begin{align*}
			|a|^p-|b|^p-p |b|^{p-2}b (a-b) 
			\leq p(p-1) |a-b|^2  \int_0^1(1-\theta) \big|b+\theta(a-b)\big|^{p-2} d\theta.
		\end{align*}
	\end{lemma}
	
	The following estimate is included in Lemma 1 in \cite{Mikulevicius2012}. 
	\begin{lemma} \label{lem:martingale_ineq}
		Let $p\geq 2$. There exists a constant $K>0$ that depends only on $p$ such that for all real-valued $\mathcal{P}\otimes\mathcal{Z}$-measurable function $f$, locally square integrable with respect to $\rho(dz) dt$,  one has
		\begin{align*}
			E \displaystyle \sup_{s\in[0,T]}\Big|\int_{0}^{t}\int_Z f(s,z) \tilde n_p(ds,dz)\Big|^{p}
			&\leq K E\Big(\int_{0}^{T}\int_Z |f(t,z)|^2 \rho(dz) dt\Big)^{p/2}
            \\
            &\qquad+ K E\int_{0}^{T}\int_Z |f(t,z)|^p \rho(dz) dt.
		\end{align*}
	\end{lemma}
	We next  outline some consequences  of the assumptions  that will be subsequently used  in the proofs of our main result and some lemmas.
	\begin{remark} \label{rem:super_linear}
		Thanks to   Assumptions \textnormal{\ref{asum:monotonocity}} and \textnormal{\ref{asum:at_0}} to \textnormal{\ref{asum:poly_lip_jump}},  there exists a constant \( K> 0 \) and a  random variable \( \Upsilon\in\mathcal{L}_{0+}^{q} \), such that,  for any $s\in[0,T]$ and $x,y\in\mathbb R^d$, 
		\begin{align*}
			& |\sigma(s,x)-\sigma(s,y)| \leq K (1+|x|+|y|)^{\zeta/2}|x-y|,
			\\
			&\int_Z |\gamma(s,x,z)-\gamma(s,y,z)|^2 \rho(dz)\leq K (1+|x|+|y|)^{\zeta}|x-y|^2,
			\\
			&|\mu(s,x)| \leq K({\Upsilon}+|x|^{\zeta+1}),  \,\,\, |\sigma(s,x)| \leq K(\Upsilon+|x|^{\zeta/2+1}),
			%\\
			%&\int_Z |\gamma(s,x,z) |^{2} \rho(dz)   \leq K ({\Upsilon}+|x|)^{\zeta+2},
			\\
			& \int_Z |\gamma(s,x,z) |^{q_0} \rho(dz)   \leq K({\Upsilon}+|x|^{\zeta+q_0})  \mbox{ for any  $2\leq q_0\leq q$.}
		\end{align*}
	\end{remark}

	\section{Moment bound of the randomized tamed Euler scheme} 
	In this section we established a moment bound (in $\mathcal{L}^p$-sense) of the randomized tamed scheme \eqref{eq:scm} (see Lemma \ref{lem:scm_mb} below). First, we prove an auxiliary result, in which  the notation $E_t(X)$ for the conditional expectation  given sigma-algebra $\mathcal{F}_t$  is used.

	\begin{lemma} \label{lem:one_step_error}
		If Assumption     \mbox{\normalfont  \ref{asum:tame}} is satisfied, then for any $2\leq q_0\leq q$ and $t\in[0,T]$,
		\begin{align*}
			E_{\kappa_n(t)}|x_t^{n}-x_{\kappa_{n}(t)}^{n}|^{q_0}\leq K
			n^{-\frac{2}{3}}\big(\Upsilon_{n}^{q_0}+|x_{\kappa_n(t)}^{n}|^{q_0}\big)
		\end{align*}
		almost surely, where  $K>0$ is a constant independent of  $n\in\mathbb N$.
	\end{lemma}
	\color{black}
	
	\begin{proof} 
		Recall   \eqref{eq:scm}, and apply H\"older's and martingale inequality along with  
		Lemma \ref{lem:martingale_ineq} to obtain
		\begin{align} 
			E_{\kappa_n(t)}|x_t^{n}-&x_{\kappa_{n}(t)}^{n}|^{q_0}			
			\leq Kn^{-{q_0}+1}E_{\kappa_n(t)}\int_{\kappa_{n}(t)}^t \big|(\widehat\mu)_{\tau m}^{\displaystyle{n}}\big(\xi_n(s),x_{\kappa_n(s)}^{n}\big) \big|^{q_0}ds \notag
			\\
			& +Kn^{-\frac{q_0}{2}+1}E_{\kappa_n(t)}\int_{\kappa_{n}(t)}^{t}\big| (\widehat\sigma)_{\tau m}^{\displaystyle{n}}(\kappa_n(s),x_{\kappa_n(s)}^{n})\big|^{q_0}  ds \notag
			\\
			&
			+Kn^{-\frac{q_0}{2}+1} E_{\kappa_n(t)}\int_{\kappa_{n}(t)}^{t}\Big(\int_Z\big|(\widehat\gamma)_{\tau m}^{\displaystyle{n}}\big(\kappa_n(s),x_{\kappa_n(s)}^{n}, z\big) \big|^{2} \rho(dz)\Big)^{q_0/2}ds \notag
			\\
			& +KE_{\kappa_n(t)}\int_{\kappa_{n}(t)}^{t}\int_Z\Big|(\widehat\gamma)_{\tau m}^{\displaystyle{n}}\big(\kappa_n(s),x_{\kappa_n(s)}^{n}, z\big)
			\Big|^{q_0}\rho(dz)ds \label{eq:onestep} 
		\end{align}
		which on  using Assumption  \ref{asum:tame}  yields
		\begin{align*}
			E_{\kappa_n(t)}|x_t^{n}-x_{\kappa_{n}(t)}^{n}|^{q_0}	
			&\leq K\big(
			n^{-\frac{2q_0}{3}}
			+n^{-\frac{q_0}{3}}
			+n^{-\frac{2}{3}}
			\big) E\big(\Upsilon_{n}+|x_{\kappa_n(t)}^{n}|\big)^{q_0}
			\\
			&\leq Kn^{-\frac{2}{3}}\big(\Upsilon_{n}^{q_0}+|x_{\kappa_n(t)}^{n}|^{q_0}\big)
		\end{align*}
		almost surely for any  $t\in[0,T]$.
		This completes the proof. 
	\end{proof}
	
	The following lemma provides a moment bound of the  scheme\eqref{eq:scm}. 
	\begin{lemma} \label{lem:scm_mb}
		If Assumptions  \mbox{\normalfont  \ref{asum:ic}},  \mbox{\normalfont  \ref{asum:coercivity_scm}}  and      \mbox{\normalfont  \ref{asum:tame}}  hold, then 
		\begin{align*}
			\sup_{ t\in[0,T]}E|x_t^{n}|^{q}\leq K
		\end{align*}
		where $K>0$ is a constant independent  of $n\in\mathbb N$.
	\end{lemma}
	\begin{proof}
		Recall the  scheme   \eqref{eq:scm}  and  employ   It\^o's formula see \cite[Theorem 32]{Protter2005}  or \cite[Theorem 94]{Situ2006}) to obtain
		\begin{align*} 
			|x_t^{n}|^{q}  
			&\leq
			|x_0|^{q}
			+{q}\int_{0}^{t}|x_s^{n}|^{{q}-2}x_s^{n}\big((\widehat\mu)_{\tau m}^{\displaystyle{n}}\big(\xi_n(s),x_{\kappa_n(s)}^{n}\big)\big) 
			ds \notag
			\\
			&\quad +{q}\int_{0}^{t}|x_s^{n}|^{{q}-2}x_s^{n}  (\widehat\sigma)_{\tau m}^{\displaystyle{n}}(\kappa_n(s),x_{\kappa_n(s)}^{n})   dw_s \notag
			\\
			&\quad +\tfrac{q(q-1)}{2}\int_{0}^{t}|x_s^{n}|^{{q}-2}\big|  (\widehat\sigma)_{\tau m}^{\displaystyle{n}}(\kappa_n(s),x_{\kappa_n(s)}^{n}) \big|^2 ds \notag
			\\
			&\quad+{q}\int_{0}^{t}  \int_Z |x_s^{n}|^{{q}-2}x_s^{n}(\widehat\gamma)_{\tau m}^{\displaystyle{n}}\big(\kappa_n(s),x_{\kappa_n(s)}^{n}, z\big)  \tilde n_p(ds,dz)  \notag
			\\
			&\quad+\int_{0}^{t}\int_Z \Big\{\big|x_s^{n}+(\widehat\gamma)_{\tau m}^{\displaystyle{n}}\big(\kappa_n(s),x_{\kappa_n(s)}^{n}, z\big)\big|^q-|x_s^{n}|^{q}
			\\
			&
			\quad -{q}\,|x_s^{n}|^{{q}-2}x_s^{n}(\widehat\gamma)_{\tau m}^{\displaystyle{n}}\big(\kappa_n(s),x_{\kappa_n(s)}^{n}, z\big)\Big\}n_p(ds,dz)
		\end{align*}
		almost surely for any $t\in[ 0, T]$, which on  taking expectation and utilizing  Lemma \ref{lem:mvt} yields  
		\begin{align}  \label{eq:mb_estimate}
			E|x_t^{n}|^{q}  \leq & E|x_0|^{q}+\frac{q}{2}E\int_{0}^{t}\Big(2 |x_{\kappa_n(s)}^{n}|^{{q}-2} x_{\kappa_n(s)}^{n}(\widehat\mu)_{\tau m}^{\displaystyle{n}}\big(\xi_n(s),x_{\kappa_n(s)}^{n}\big) \notag
			\\
			&\qquad\qquad\quad+(q-1)|x_{\kappa_n(s)}^{n}|^{{q}-2}\big| (\widehat\sigma)_{\tau m}^{\displaystyle{n}}(\kappa_n(s),x_{\kappa_n(s)}^{n})\big|^2 \notag
			\\
			&\qquad\qquad\quad
			+2(q-1)\int_Z\big| (\widehat\gamma)_{\tau m}^{\displaystyle{n}}\big(\kappa_n(s),x_{\kappa_n(s)}^{n}, z\big) \big|^2
			\notag
			\\
			&\qquad\qquad\qquad
			\times\int_{0}^1 (1-\theta)  \big |x_{\kappa_{n}(s)}^{n}+\theta(\widehat\gamma)_{\tau m}^{\displaystyle{n}}\big(\kappa_n(s),x_{\kappa_n(s)}^{n}, z\big) \big|^{q-2} d\theta \rho(dz)  \notag\Big) ds
			\notag
			\\
			&+ {q}E\int_{0}^{t}|x_{s}^{n}|^{{q}-2}\big(x_s^{n}-x_{\kappa_n(s)}^{n}\big)(\widehat\mu)_{\tau m}^{\displaystyle{n}}\big(\xi_n(s),x_{\kappa_n(s)}^{n}\big)  ds \notag
			\\
			&+{q}E\int_{0}^{t}\big(|x_s^{n}|^{{q}-2}-|x_{\kappa_n(s)}^{n}|^{{q}-2}\big)x_{\kappa_n(s)}^{n}(\widehat\mu)_{\tau m}^{\displaystyle{n}}\big(\xi_n(s),x_{\kappa_n(s)}^{n}\big)  ds \notag
			%			\\
			%			&+qE\int_{0}^{t}|x_{\kappa_n(s)}^{n}|^{{q}-2}x_{\kappa_n(s)}^{n}\big((\widehat\mu)_{\tau m}^{\displaystyle{n}}\big(\xi_n(s),x_{\kappa_n(s)}^{n}\big)-(\widehat\mu)_{\tau m}^{\displaystyle{n}}\big(\xi_n(s),x_{\kappa_n(s)}^{n}\big)s \big) ds \notag
			\\
			&+\tfrac{q(q-1)}{2}E\int_{0}^{t}\big
			(|x_s^{n}|^{{q}-2}-|x_{\kappa_n(s)}^{n}|^{{q}-2}\big)\big| (\widehat\sigma)_{\tau m}^{\displaystyle{n}}(\kappa_n(s),x_{\kappa_n(s)}^{n})\big|^2  ds \notag
			%			\\
			%			&\quad+E\int_{0}^{t}\int_Z \Big(\big|x_s^{n}+(\widehat\gamma)_{\tau m}^{\displaystyle{n}}\big(\kappa_n(s),x_{\kappa_n(s)}^{n}, z\big)\big|^{q}-|x_s^{n}|^{q} \notag
			%			\\
			%			&\qquad\qquad\qquad-{q}\,|x_s^{n}|^{{q}-2}x_s^{n}(\widehat\gamma)_{\tau m}^{\displaystyle{n}}\big(\kappa_n(s),x_{\kappa_n(s)}^{n}, z\big) -q(q-1)\big| (\widehat\gamma)_{\tau m}^{\displaystyle{n}}\big(\kappa_n(s),x_{\kappa_n(s)}^{n}, z\big) \big|^2 \notag
			%			\\
			%			&\qquad\qquad\qquad\qquad \times\int_{0}^1 (1-\theta)  \big |x_{\kappa_{n}(s)}^{n}+\theta(\widehat\gamma)_{\tau m}^{\displaystyle{n}}\big(\kappa_n(s),x_{\kappa_n(s)}^{n}, z\big) \big|^{q-2} d\theta \Big)\rho(dz)ds \notag
			%	\\
			\\
			&+ q(q-1)E\int_{0}^{t}\int_Z \big| (\widehat\gamma)_{\tau m}^{\displaystyle{n}}\big(\kappa_n(s),x_{\kappa_n(s)}^{n}, z\big) \big|^2\notag
			\\
			& \qquad\qquad\qquad 
			\times\int_{0}^1 (1-\theta) \Big( \big |x_s^{n}+\theta(\widehat\gamma)_{\tau m}^{\displaystyle{n}}\big(\kappa_n(s),x_{\kappa_n(s)}^{n}, z\big) \big|^{q-2}  \notag
			\\
			& \qquad\qquad\qquad\qquad  \qquad       - \big |x_{\kappa_{n}(s)}^{n}+\theta(\widehat\gamma)_{\tau m}^{\displaystyle{n}}\big(\kappa_n(s),x_{\kappa_n(s)}^{n}, z\big) \big|^{q-2}\Big) d\theta\rho(dz)ds\notag
			\\
			& \qquad \qquad \leq  E|x_0|^{q}+E(\widehat{\Upsilon}_{n}^{q})+K\int_0^t\sup_{ r\in[0,s]}E|x_r^{n}|^{q}ds+\sum_{\mathfrak{u}=1}^{4} \Psi_{\mathfrak{u}}
		\end{align}
		for any $t\in[ 0, T],$ 	where  the {first time integral was bounded  by using Assumption }\ref{asum:coercivity_scm}. 
		
		Further, we need the following inequality, a consequence of Lemma~\ref{lem:mvt} for $q \geq 4$,
		\begin{align} \label{eq:reminder_formula}
			|a|^{q-2}&- |b|^{q-2}
			\leq K \big(|b|^{q-3}+|a-b|^{q-3}\big) |a-b|
		\end{align}
		for any $a, b \in \mathbb{R}^d$. To upper bound $\Psi_1+\Psi_2$,  utilize   \eqref{eq:reminder_formula} and  Assumption  \ref{asum:tame} together with Young's inequality and Lemma \ref{lem:one_step_error}  as follows 
		\begin{align}
			\Psi_1+&\Psi_2 :={q}E\int_{0}^{t}|x_{s}^{n}|^{{q}-2}\big(x_s^{n}-x_{\kappa_n(s)}^{n}\big)(\widehat\mu)_{\tau m}^{\displaystyle{n}}\big(\xi_n(s),x_{\kappa_n(s)}^{n}\big)  ds \notag
			\\
			&\qquad+{q}E\int_{0}^{t}\big(|x_s^{n}|^{{q}-2}-|x_{\kappa_n(s)}^{n}|^{{q}-2}\big)x_{\kappa_n(s)}^{n}(\widehat\mu)_{\tau m}^{\displaystyle{n}}\big(\xi_n(s),x_{\kappa_n(s)}^{n}\big)  ds \notag
			\\
			&\leq   KE\int_{0}^{t} \big(|x_{\kappa_n(s)}^{n}|^{{q}-2}+|x_s^{n}-x_{\kappa_n(s)}^{n}|^{{q}-2}\big)\big|x_s^{n}-x_{\kappa_n(s)}^{n}\big|\big|(\widehat\mu)_{\tau m}^{\displaystyle{n}}\big(\xi_n(s),x_{\kappa_n(s)}^{n}\big)\big| ds \notag
			\\
			&\qquad
			+ KE\int_{0}^{t}\big(|x_{\kappa_n(s)}^{n}|^{q-3}+|x_s^{n}-x_{\kappa_n(s)}^{n}|^{q-3}\big)\notag
			\\
			&\qquad\times \big|x_s^{n}-x_{\kappa_n(s)}^{n}\big|\big|x_{\kappa_n(s)}^{n}\big| \big|(\widehat\mu)_{\tau m}^{\displaystyle{n}}\big(\xi_n(s),x_{\kappa_n(s)}^{n}\big)\big|ds \notag
			\\
			&\leq  K E\int_{0}^{t} n^{\frac{1}{3}}\big(\Upsilon_{n}+|x_{\kappa_n(s)}^{n}|\big)\Big(|x_{\kappa_n(s)}^{n}|^{q-3}n^{-\frac{1}{6}} |x_{\kappa_n(s)}^{n}| n^{\frac{1}{6}} |x_s^{n}-x_{\kappa_n(s)}^{n}| 
			\notag
			\\
			& \qquad + |x_s^{n}-x_{\kappa_n(s)}^{n}|^{q-1}  + |x_{\kappa_n(s)}^{n}||x_s^{n}-x_{\kappa_n(s)}^{n}|^{q-2}  \notag
			\Big)  ds  \notag
			\\
			&\leq  K E\int_{0}^{t} n^{\frac{1}{3}}\big(\Upsilon_{n}+|x_{\kappa_n(s)}^{n}|\big) \Big(|x_{\kappa_n(s)}^{n}|^{q-3}n^{-\frac{1}{3}}|x_{\kappa_n(s)}^{n}|^{2}  \notag
			\\
			&\qquad+ |x_{\kappa_n(s)}^{n}|^{q-3}n^{\frac{1}{3}}E_{\kappa_{n}(s)}|x_s^{n}-x_{\kappa_n(s)}^{n}| ^{2}+ E_{\kappa_{n}(s)}|x_s^{n}-x_{\kappa_n(s)}^{n}|^{q-1} \notag
			\\
			&\qquad+ |x_{\kappa_n(s)}^{n}|E_{\kappa_{n}(s)}|x_s^{n}-x_{\kappa_n(s)}^{n}|^{q-2} \Big)   ds   \notag
			\\
			&\leq Kn^{\frac{1}{3}} \big(n^{-\frac{1}{3}}+n^{-\frac{2}{3}}\big)E\int_{0}^t\big(\Upsilon_{n}^{q}+|x_{\kappa_n(s)}^{n}|^{q}\big) ds\leq K\int_{0}^t\big(1+\sup_{ r\in[0,s]}E|x_r^{n}|^{q}\big)ds\label{eq:Psi_1_2}
		\end{align}
		for any 	$t\in[ 0, T]$.
		To upper bound $\Psi_3$, we  use   \eqref{eq:reminder_formula} and  Assumption  \ref{asum:tame} along with Young's inequality and Lemma \ref{lem:one_step_error}  to have the following
		\begin{align}\label{eq:Psi_4}
			\Psi_3 :=&\tfrac{q(q-1)}{2}E\int_{0}^{t}\big(|x_s^{n}|^{{q}-2}-|x_{\kappa_n(s)}^{n}|^{{q}-2}\big)\big| (\widehat\sigma)_{\tau m}^{\displaystyle{n}}(\kappa_n(s),x_{\kappa_n(s)}^{n})\big|^2  ds \notag
			\\
			\leq& KE\int_{0}^{t}\big(|x_{\kappa_n(s)}^{n}|^{q-3}+|x_s^{n}-x_{\kappa_n(s)}^{n}|^{q-3}\big) \big|x_s^{n}-x_{\kappa_n(s)}^{n}\big| \big| (\widehat\sigma)_{\tau m}^{\displaystyle{n}}(\kappa_n(s),x_{\kappa_n(s)}^{n})\big|^2  ds  \notag
			\\
			\leq &  K E\int_{0}^{t} \left(|x_{\kappa_n(s)}^{n}|^{q-4} n^{-\frac{1}{6}} |x_{\kappa_n(s)}^{n}|n^{\frac{1}{6}} |x_s^{n}-x_{\kappa_n(s)}^{n}| 
			+  |x_s^{n}-x_{\kappa_n(s)}^{n}|^{q-2} \right) \notag
            \\
            &\qquad\qquad\qquad\qquad\qquad\times\big| (\widehat\sigma)_{\tau m}^{\displaystyle{n}}(\kappa_n(s),x_{\kappa_n(s)}^{n})\big|^2   ds  \notag
			\\
			\leq &  K E\int_{0}^{t} n^{\frac{1}{3}}\big(\Upsilon_{n}+|x_{\kappa_n(s)}^{n}|\big)^2\Big( |x_{\kappa_n(s)}^{n}|^{q-4}n^{-\frac{1}{3}}|x_{\kappa_n(s)}^{n}|^{2}\notag
			\\
			&  + |x_{\kappa_n(s)}^{n}|^{q-4}n^{\frac{1}{3}}E_{\kappa_{n}(s)}|x_s^{n}-x_{\kappa_n(s)}^{n}| ^{2}+E_{\kappa_{n}(s)}|x_s^{n}-x_{\kappa_n(s)}^{n}|^{q-2} \Big) ds  \notag
			\\
			\leq &K n^{\frac{1}{3}}(n^{-\frac{1}{3}}+n^{-\frac{2}{3}}) E\int_{0}^{t}\big(\Upsilon_{n}^{q}+|x_{\kappa_n(s)}^{n}|^{q}\big) ds 
			\leq   K  \int_{0}^t\big(1+\sup_{ r\in[0,s]}E|x_r^{n}|^{q}\big)ds
		\end{align}
		for any $t\in[ 0, T]$.
		For $\Psi_4$, one utilizes
		%	Lemma \ref{lem:mvt} and
		\eqref{eq:reminder_formula}  to obtain 
		\begin{align*}
			\Psi_4 &:=
			%			E\int_{0}^{t}\int_Z \Big(\big|x_s^{n}+(\widehat\gamma)_{\tau m}^{\displaystyle{n}}\big(\kappa_n(s),x_{\kappa_n(s)}^{n}, z\big)\big|^{q}-|x_s^{n}|^{q} -{q}\,|x_s^{n}|^{{q}-2}x_s^{n}(\widehat\gamma)_{\tau m}^{\displaystyle{n}}\big(\kappa_n(s),x_{\kappa_n(s)}^{n}, z\big)\notag
			%			\\
			%			&\quad-q(q-1)\big| (\widehat\gamma)_{\tau m}^{\displaystyle{n}}\big(\kappa_n(s),x_{\kappa_n(s)}^{n}, z\big) \big|^2 
			%			\int_{0}^1 (1-\theta)  \big |x_{\kappa_{n}(s)}^{n}+\theta(\widehat\gamma)_{\tau m}^{\displaystyle{n}}\big(\kappa_n(s),x_{\kappa_n(s)}^{n}, z\big) \big|^{q-2} d\theta \Big)\rho(dz)ds \notag
			%			\\
			%			& \leq q(q-1)E\int_{0}^{t}\int_Z\Big\{\big| (\widehat\gamma)_{\tau m}^{\displaystyle{n}}\big(\kappa_n(s),x_{\kappa_n(s)}^{n}, z\big) \big|^2 \notag
			%			\\
			%			&\quad \times  \int_{0}^1 (1-\theta)\Big |x_s^{n}+\theta(\widehat\gamma)_{\tau m}^{\displaystyle{n}}\big(\kappa_n(s),x_{\kappa_n(s)}^{n}, z\big) \Big|^{q-2} d\theta \notag 
			%			\\
			%			&\quad- \big| (\widehat\gamma)_{\tau m}^{\displaystyle{n}}\big(\kappa_n(s),x_{\kappa_n(s)}^{n}, z\big) \big|^2
			%			\int_{0}^1 (1-\theta)  \big |x_{\kappa_{n}(s)}^{n}+\theta(\widehat\gamma)_{\tau m}^{\displaystyle{n}}\big(\kappa_n(s),x_{\kappa_n(s)}^{n}, z\big) \big|^{q-2} d\theta\Big\} \rho(dz)ds  \notag
			%			\\
			%			&\leq 
			q(q-1)E\int_{0}^{t}\int_Z \big| (\widehat\gamma)_{\tau m}^{\displaystyle{n}}\big(\kappa_n(s),x_{\kappa_n(s)}^{n}, z\big) \big|^2
            \notag
			\\
			&\qquad\qquad\times
			\int_{0}^1 (1-\theta) \Big( \big |x_s^{n}+\theta(\widehat\gamma)_{\tau m}^{\displaystyle{n}}\big(\kappa_n(s),x_{\kappa_n(s)}^{n}, z\big) \big|^{q-2}  \notag
			\\
			& \qquad\qquad\qquad\qquad         - \big |x_{\kappa_{n}(s)}^{n}+\theta(\widehat\gamma)_{\tau m}^{\displaystyle{n}}\big(\kappa_n(s),x_{\kappa_n(s)}^{n}, z\big) \big|^{q-2}\Big) d\theta\rho(dz)ds\notag
			\\
			&\leq  KE\int_{0}^{t}\int_Z \big| (\widehat\gamma)_{\tau m}^{\displaystyle{n}}\big(\kappa_n(s),x_{\kappa_n(s)}^{n}, z\big) \big|^{2}  \int_{0}^{1} \Big(\big| x_{\kappa_{n}(s)}^{n}+\theta(\widehat\gamma)_{\tau m}^{\displaystyle{n}}\big(\kappa_n(s),x_{\kappa_n(s)}^{n}, z\big)\big|^{q-3}  \notag
			\\
			&\qquad \qquad\qquad\qquad +\big|x_s^{n}-x_{\kappa_{n}(s)}^{n} \big|^{q-3} \Big)  
			\big|x_s^{n}-x_{\kappa_{n}(s)}^{n}\big| d\theta\rho(dz)ds \notag
			\\
			&\leq KE\int_{0}^{t}\int_Z   \Big( \big| (\widehat\gamma)_{\tau m}^{\displaystyle{n}}\big(\kappa_n(s),x_{\kappa_n(s)}^{n}, z\big) \big|^2  |x_{\kappa_{n}(s)}^{n}|^{q-3} |x_s^{n}-x_{\kappa_{n}(s)}^{n}|\notag
			\\
			&\qquad\qquad\qquad\qquad +\big| (\widehat\gamma)_{\tau m}^{\displaystyle{n}}\big(\kappa_n(s),x_{\kappa_n(s)}^{n}, z\big) \big|^{q-1} |x_s^{n}-x_{\kappa_{n}(s)}^{n}|
			\\
            &\qquad\qquad\qquad\qquad +\big| (\widehat\gamma)_{\tau m}^{\displaystyle{n}}\big(\kappa_n(s),x_{\kappa_n(s)}^{n}, z\big) \big|^{2} |x_s^{n}-x_{\kappa_{n}(s)}^{n}\big|^{q-2} \Big) \rho(dz)ds
		\end{align*}
		which on applying Assumption \ref{asum:tame}, Young's inequality and Lemma \ref{lem:one_step_error} yields 
		\begin{align}
			\Psi_4
			&\leq  KE\int_{0}^{t} \Big(
			n^{\frac{1}{3}}\big(\Upsilon+|x_{\kappa_n(s)}^{n}|\big)^2 |x_{\kappa_{n}(s)}^{n}|^{q-4} \,n^{-\frac{1}{6}}\,|x_{\kappa_n(s)}^{n}|  \,n^{\frac{1}{6}}\,|x_s^{n}-x_{\kappa_n(s)}^{n}|
			\notag
			\\
			&\quad +n^{\frac{1}{3}}\big(\Upsilon\!+\!|x_{\kappa_n(s)}^{n}|\big)^{q-1}n^{-\frac{1}{6}}\,n^{\frac{1}{6}} |x_s^{n}\!-\!x_{\kappa_{n}(s)}^{n}| \!+ \! n^{\frac{1}{3}}\big(\Upsilon\!+\!|x_{\kappa_n(s)}^{n}|\big)^{2} |x_s^{n}\!-\!x_{\kappa_{n}(s)}^{n}\big|^{q-2} \Big) ds \notag \notag
			\\
			&\leq  KE\int_{0}^{t} 
			n^{\frac{1}{3}}\big(\Upsilon+|x_{\kappa_n(s)}^{n}|\big)^2 |x_{\kappa_{n}(s)}^{n}|^{q-4} \big(n^{-\frac{1}{3}}|x_{\kappa_n(s)}^{n}|^{2} + n^{\frac{1}{3}}E_{\kappa_{n}(s)}|x_s^{n}-x_{\kappa_n(s)}^{n}| ^{2}\big) ds \notag
			\\
			&\quad+  \!KE\int_{0}^{t} \!
			n^{\frac{1}{3}}\big(\Upsilon+|x_{\kappa_n(s)}^{n}|\big)^{q-2} 
			\Big(n^{-\frac{1}{3}}\big(\Upsilon+|x_{\kappa_n(s)}^{n}|\big)^{2} + n^{\frac{1}{3}}E_{\kappa_{n}(s)}|x_s^{n}-x_{\kappa_n(s)}^{n}| ^{2}\Big)ds \notag
			\\
			&\quad+K n^{\frac{1}{3}} E\int_{0}^{t}\big(\Upsilon+|x_{\kappa_n(s)}^{n}|\big)^{2} E_{\kappa_{n}(s)}|x_s^{n}-x_{\kappa_{n}(s)}^{n}\big|^{q-2} ds \notag
			\\
			&\leq Kn^{\frac{1}{3}}(n^{-\frac{1}{3}}+n^{-\frac{2}{3}})E\int_{0}^{t}\big(\Upsilon^q+|x_{\kappa_n(s)}^{n}|^{q}\big) ds\leq  K \int_{0}^t\big(1+\sup_{ \in[0,s]}E|x_r^{n}|^{q}\big)ds \label{eq:Psi_5}
		\end{align}
		for any $t\in[0,T]$.
		Then, by employing  \eqref{eq:Psi_1_2} to   \eqref{eq:Psi_5} in  \eqref{eq:mb_estimate} for any  $t\in[0,T]$,
		\begin{align*}
			\sup_{ r\in[0,t]} E|x_r^{n}|^{q}\leq E|x_0|^{q} +K+ K \int_{0}^t\sup_{ r\in[0,s]}E|x_r^{n}|^{q} ds
		\end{align*}
		which on applying Assumption \ref{asum:ic} and Gr\"onwall's inequality  completes the proof.
	\end{proof} 
		\section{Rate of Convergence} \label{sec:convergence_rate}
	{The main part of Theorem \ref{thm:main_result}, namely the convergence rate, is established in the sequel.
		{The next corollary constitutes a first step in that direction, and follows as an immediate consequence of Remark \ref{rem:super_linear} and Proposition \ref{prop:mb:sde}.  }}
	\begin{corollary} \label{cor:one_step_regularity_sde}
		If Assumptions \mbox{\normalfont  \ref{asum:ic}}, \mbox{\normalfont \ref{asum:monotonocity}},  \mbox{\normalfont \ref{asum:coercivity_p}}, and \mbox{\normalfont \ref{asum:at_0}}  to \textnormal{\ref{asum:poly_lip_jump}}   are satisfied,
		then 
		\begin{align*}
			E|x_s-x_{\kappa_{n}(s)}|^{q_0}\leq Kn^{-1}
		\end{align*}
		for any   $s\in[0,T]$ and $q_0\geq 2$ with $  q_0{(\zeta+1)}\leq {q}$,	where $K>0$    is a constant independent of  $n\in \mathbb N$. 
	\end{corollary}
	
	{ Further, we obtain the following thanks to Hypothesis \ref{asum:tame} and Lemma \ref{lem:scm_mb}}.
	\begin{corollary} %\label{cor:_coefficients_mb}
		If Assumptions \mbox{\normalfont  \ref{asum:ic}}, and \mbox{\normalfont  \ref{asum:coercivity_scm}}  to \textnormal{\ref{asum:poly_lip_jump}} are satisfied,
		then 
		\begin{align*}
			E|(\widehat\mu)_{\tau m}^{\displaystyle{n}}\big(\xi_n(s)&,x_{\kappa_n(s)}^{n}\big)|^{q_0}+E| (\widehat\sigma)_{\tau m}^{\displaystyle{n}}(\kappa_n(s),x_{\kappa_n(s)}^{n})|^{q_0}
            \\
            &+E\int_Z |(\widehat\gamma)_{\tau m}^{\displaystyle{n}}\big(\kappa_n(s),x_{\kappa_n(s)}^{n}, z\big)|^{q_0} \rho(dz) \leq K
		\end{align*}
		for any   $s\in[0,T]$ and $q_0\geq 2$ with $  q_0{(\zeta+1)}\leq {q}$,	where $K>0$    is a constant independent of  $n\in \mathbb N$. 
	\end{corollary}

	{The following corollary is a consequence of \eqref{eq:onestep} and the above corollary.}
	\begin{corollary} \label{cor:one_step_error}
		If Assumptions \mbox{\normalfont  \ref{asum:ic}}, and \mbox{\normalfont  \ref{asum:coercivity_scm}}  to \textnormal{\ref{asum:poly_lip_jump}}  are satisfied,
		then 
		\begin{align*}
			E|x_s^{n}-x_{\kappa_{n}(s)}^{n}|^{q_0}\leq Kn^{-1}
		\end{align*}
		for any   $s\in[0,T]$ and $q_0\geq 2$ with $  q_0{(\zeta+1)}\leq {q}$,	where $K>0$    is a constant independent of  $n\in \mathbb N$. 
	\end{corollary}

	%\medskip
	
	%\noindent \textbf{Remark.} The independence between \( \varphi \) and \( \mathcal{F}_\infty^1 \) is essential. Without it, the conditional distribution of \( x_{\xi_n(s)}^{n, \, \varphi} \) given \( \mathcal{F}_\infty \) cannot be described solely in terms of the filtration \( \mathbb{F}^{\varphi} \), and the adaptedness of \( x_{\xi_n(s)}^{n, \, \varphi} \) to \( \mathbb{F}^{\varphi} \) may no longer hold.
	To establish the rate of convergence for the randomized tamed Euler scheme~\eqref{eq:scm} 
	applied to the Carath\'eodory-type SDE~\eqref{eq:sde}, the classical approaches developed for the randomized and tamed Euler schemes in \cite{Przybylowicz2015a, Przybylowicz2015b, Przybylowicz2014, Przybylowicz2022} and \cite{Chen2019, Dareiotis2016, Kumar2017a, Neelima2020}, respectively, cannot be employed directly.
	This limitation necessitates extending these existing methodologies and developing a new analytical framework. As a key step, 
	let us introduce the following auxiliary equation for any \( t \in [0,T], \)
	\begin{align}
		y_t^n 
		&= x_0 
		+ \int_0^t \mu(\xi_n(s), x_{\kappa_n(s)}) \, ds 
		+ \int_0^t \sigma(s, x_s) \, dw_s 
		+ \int_0^t \int_Z \gamma(s, x_s, z) \, \tilde{n}_p(ds,dz).
		\label{eq:auxiliary_equation}
	\end{align}
	\begin{remark} \label{rem:mb_auxilary}
		Due to Remark \ref{rem:super_linear} and Proposition \ref{prop:mb:sde}\, $\sup_{t\in[0,T]}E|y_t^n|^{q_0}\leq K$, for any $q_0> 0$ with $  q_0{(\zeta+1)}\leq {q}$, where $K>0$ is a constant independent of $n\in\mathbb N$.
	\end{remark}
	The specific  technical challenges of our framework are tackled through a two-fold strategy: one component, relying on a randomization technique, is addressed in Lemma~\ref{lem:estimate_first_term_MR}, while the other, based on a taming approach, is presented in Lemma~\ref{lem:estimate_last_term_MR}.
	
	To facilitate the forthcoming analysis, define the sigma-algebra \( \tilde{\mathcal{F}}_\infty \) as the one generated by the filtration \( (\tilde{\mathcal{F}}_t)_{t \geq 0} \), which encompasses both the Brownian motion and the Poisson random measure. Additionally, recall that the sequence \( \varphi := (\varphi_{k-1})_{k \in \mathbb{N}} \) is independent of \( \tilde{\mathcal{F}}_\infty \).
	This independence implies that, conditional on \( \tilde{\mathcal{F}}_\infty \), the process \( (x_{t})_{t\geq 0} \) is adapted to the filtration \( \mathbb{F}^{\varphi} \), under which \( \varphi \) is defined. 
	\begin{lemma} \label{lem:estimate_first_term_MR}
		Let Assumptions \mbox{\normalfont  \ref{asum:ic}}, \mbox{\normalfont \ref{asum:monotonocity}} and \mbox{\normalfont \ref{asum:coercivity_p}}, \mbox{\normalfont \ref{asum:at_0}}  to \textnormal{\ref{asum:poly_lip_jump}} hold.
		Then,
		\begin{align*}
			&E|x_t-y_t^{n}|^{q_0}
			\leq  K n^{-\frac{q_0}{q_0+\delta}}
		\end{align*}
		for any   $t\in[0,T]$ and $q_0\geq 2$ with $\max\big\{q_0\delta^{-1}(q_0+\delta)\zeta,(q_0+\delta)(\zeta+1)\big\}\leq q$, $\delta\in(0,1)$,	where $K>0$    is a constant independent of  $n\in \mathbb N$. 
		\begin{proof}
			In the proof we use the following notation for any $t\in[0, T]:$
			\[
			k(t) := \sup\!\Big\{k=0,\ldots,n\,\big|\;\tfrac{kT}{n}\le t\Big\},
			\quad 
			\varrho(t) := \tfrac{k(t)T}{n}.
			\]
			Notice that, for any $k\in\{1,\ldots, n\}$  one has: 
			\begin{align*}
				E\Big(\int_{t_{k-1}}^{t_k}\mu(\xi_n(s), x_{\kappa_n(s)}) ds \,\Big|\, \tilde{\mathcal F}_{\infty}\Big)
				&=E^{\varphi} \int_{t_{k-1}}^{t_k} \mu({t_{k-1}+\Delta t\varphi_{k-1}}, x_{t_{k-1}})  ds
				\\
				&=\int_0^1 \mu({t_{k-1}+\Delta t r},x_{t_{k-1}}) dr  \Delta t
				\\
				&
				%= \int_{t_{k-1}}^{t_{k}} \mu(s,x_s^{n\, \varphi}) ds
				= \int_{t_{k-1}}^{t_{k}} 
				\mu(s,x_{\kappa_n(s)}) ds.
			\end{align*}
			Then, applying conditionally on $\tilde{\mathcal F}_{\infty}$  the first bound on randomized Riemann sum approximations stated in    \cite[Theorem 4.1]{Kruse2019},  together with Remark \ref{rem:super_linear} and Proposition \ref{prop:mb:sde}, for any $t\in [0,T]$ we get 
			\begin{align} \label{eq:term_1}
&E\bigg|\sum_{k=1}^{k(t)}\int_{t_{k-1}}^{t_k}\big(\mu(s, x_{\kappa_n(s)})-\mu(\xi_n(s), x_{\kappa_n(s)})\big)ds\bigg|^{q_0} \notag
				\\
				&\qquad \leq E\left(E\bigg[\max_{j\in\{1, \ldots, n\}}\Big|\sum_{k=1}^{j}\int_{t_{k-1}}^{t_k}\big(\mu(s, x_{\kappa_n(s)})-\mu(\xi_n(s), x_{\kappa_n(s)})\big)ds\Big|^{q_0}\bigg|\tilde{\mathcal F}_{\infty}\bigg] \right) \notag
				\\
				&\qquad\leq K\sup_{t\in[0,T]}E|\mu(t, x_{t})|^{q_0}n^{-\frac{q_0}{2}}	\leq K\big(E(\Upsilon^{q_0})+\sup_{t\in[0,T]}E|x_t|^{(\zeta+1)q_0}\big)n^{-\frac{q_0}{2}}\leq K n^{-\frac{q_0}{2}}.
			\end{align}
			Then,  use  \eqref{eq:sde}, \eqref{eq:auxiliary_equation}, \eqref{eq:term_1} along with H\"older's inequality and Remark \ref{rem:super_linear}  to obtain
			\begin{align} %\label{eq:dif:scm_auxlr}
				E|x_t-y_t^{n}|^{q_0}=&E\Big|\int_{0}^{t}\big(\mu(s, x_{s})-\mu(\xi_n(s), x_{\kappa_n(s)})\big)ds\Big|^{q_0} \notag
				\\
				\leq & 
				KE\Big|\int_{0}^{t}\big(\mu(s, x_{s})-\mu(s, x_{\kappa_n(s)})\big)ds\Big|^{q_0} \notag
				\\
				&\qquad+ K E\bigg|\sum_{k=1}^{k(t)}\int_{t_{k-1}}^{t_k}\big(\mu(s, x_{\kappa_n(s)})-\mu(\xi_n(s), x_{\kappa_n(s)})\big)ds\bigg|^{q_0}\notag
				\\
				&\qquad+ K E\bigg|\int_{\varrho(t)}^{t}\big(\mu(s, x_{\kappa_n(s)})-\mu(\xi_n(s), x_{\kappa_n(s)})\big)ds\bigg|^{q_0} \notag
				\\
				\leq &	KE\int_{0}^{t}\big(1+|x_{s}|^{\zeta q_0}+|x_{\kappa_n(s)}|^{\zeta q_0}\big)|x_{s}-x_{\kappa_n(s)}|^{q_0}ds +Kn^{-\frac{q_0}{2}} \notag
				\\
				&\qquad+ K (t-\varrho(t))^{q_0-1}\int_{\varrho(t)}^{t}\sup_{r\in[0,T]}E|\mu(r, x_{\kappa_n(r)})|^{q_0}ds \notag
				\\
				\leq & K\int_{0}^{t}\!\big(E\big(1+|x_{\kappa_n(s)}|^{\zeta q_0}+|x_{s}|^{\zeta q_0}\big)^{\frac{q_0+\delta}{\delta}}\big)^{\frac{\delta}{q_0+\delta}}\big(E|x_{s}-x_{\kappa_n(s)}|^{q_0+\delta}\big)^{\frac{q_0}{q_0+\delta}}ds \notag
				\\
				&\qquad +Kn^{-\frac{q_0}{2}} 
				+Kn^{-q_0}\sup_{r\in[0,T]}E\big(\Upsilon^{q_0}+E|x_{r}|^{(\zeta+1)q_0}\big) \notag
			\end{align}
			for any $t\in[0,T]$, which on using  Proposition \ref{prop:mb:sde} and Corollary \ref{cor:one_step_regularity_sde} completes the proof.
		\end{proof}
	\end{lemma}

	The following lemma provides an estimate crucial for proving Lemma~\ref{lem:estimate_last_term_MR} and for establishing the  convergence rate of the scheme~\eqref{eq:scm}.
	\begin{lemma} \label{lem:estimate_tame_coeff_MR}
		If  Assumptions \mbox{\normalfont  \ref{asum:ic}}, \mbox{\normalfont  \ref{asum:monotonocity}} and \mbox{\normalfont  \ref{asum:coercivity_p}} 
		to \mbox{\normalfont  \ref{asum:convergence}}  are satisfied.
		Then, 
		\begin{align*}
			&E\big|\mu(\xi_n\!(s),x_{\kappa_n(s)})\!-\!\mu( \xi_n\!(s), y_{s}^n)\big|^{q_0}\!\!+\!E\big|\mu(\xi_n\!(s),x_{s}^n)\!-\!(\widehat\mu)_{\tau m}^{\displaystyle{n}}\!\big(\xi_n\!(s),x_{\kappa_n(s)}^{n}\big)\big|^{q_0} \!\!\leq \!K \! n^{\!-\frac{q_0}{q_0+\delta}},
			\\
			&E\big|\sigma(s,x_{s})-\sigma(s, y_{s}^n)\big|^{q_0}
			+ E\big|\sigma(s, x_{s}^n)- (\widehat\sigma)_{\tau m}^{\displaystyle{n}}(\kappa_n(s),x_{\kappa_n(s)}^{n})\big|^{q_0} \ \leq Kn^{-\min\{\frac{q_0}{q_0+\delta},\,\alpha q_0\}},
			\\
			&E\int_Z \big|\gamma(s,x_{s},z)-\gamma(s, y_{s}^n,z)\big|^{q_0}\rho(dz)
			\\
			&\qquad\qquad\,\,\,+E\int_Z \big|\gamma(s, x_{s}^n,z)-(\widehat\gamma)_{\tau m}^{\displaystyle{n}}\big(\kappa_n(s),x_{\kappa_n(s)}^{n}, z\big)\big|^{q_0} \rho(dz) \leq Kn^{-\min\{\frac{q_0}{q_0+\delta},\,\beta q_0\}}
		\end{align*}
		for any  $q_0\geq 2$ and $\delta\in(0,1)$ such that $\max\big\{q_0\delta^{-1}(q_0+\delta)\zeta,(q_0+\delta)(\zeta+1)\big\}\leq q$ and all $s\in [0,T]$,  where  $K>0$ is a constant independent of $n\in\mathbb N$.
	\end{lemma}	
	\begin{proof}
		To obtain the first inequality, for any $s\in[0,T]$, use Assumptions \ref{asum:poly_lip_drift} and \ref{asum:convergence} to get 
		\begin{align}
			E\big|\mu(&\xi_n(s),x_{\kappa_n(s)})-\mu( \xi_n(s), y_{s}^n)\big|^{q_0}+E\big|\mu(\xi_n(s),x_{s}^n)-(\widehat\mu)_{\tau m}^{\displaystyle{n}}\big(\xi_n(s),x_{\kappa_n(s)}^{n}\big)\big|^{q_0} \notag
			\\
			&\leq  KE\big|\mu(\xi_n(s),x_{\kappa_n(s)})-\mu( \xi_n(s), x_{s})\big|^{q_0} + KE\big|\mu(\xi_n(s),x_{s})-\mu( \xi_n(s), y_{s}^n)\big|^{q_0} \notag
			\\
			&\quad +KE\big|\mu(\xi_n(s), x_{s}^{n})-\mu(\xi_n(s), x_{\kappa_n(s)}^{n})\big|^{q_0}ds  \notag
			\\
			&\quad +KE\big|\mu(\xi_n(s), x_{\kappa_n(s)}^{n})-(\widehat\mu)_{\tau m}^{\displaystyle{n}}\big(\xi_n(s),x_{\kappa_n(s)}^{n}\big)\big|^{q_0} \notag
			\\
			&\leq 	KE\big((1+|x_{\kappa_n(s)}|^{\zeta q_0}+|x_{s}|^{\zeta q_0})|x_{s}-x_{\kappa_n(s)}|^{q_0}\big) \notag
			\\
			&\quad +KE\big((1+|x_{s}|^{\zeta q_0}+|y_{s}^n|^{\zeta q_0})|x_{s}-y_{s}^n|^{q_0} \big) \notag
			\\
			&\quad+ KE\big((1+|x_{s}^n|^{\zeta q_0}+|x_{\kappa_n(s)}^n|^{\zeta q_0})|x_{s}^n-x_{\kappa_n(s)}^n|^{q_0}\big)+ K n^{-\frac{q_0}{q_0+\delta}} \notag
			\\
			&\leq   K n^{-\frac{q_0}{q_0+\delta}}+K\big(E\big(1+|x_{\kappa_n(s)}|^{\zeta q_0}+|x_{s}|^{\zeta q_0}\big)^{\frac{q_0+\delta}{\delta}}\big)^{\frac{\delta}{q_0+\delta}}\big(E|x_{s}-x_{\kappa_n(s)}|^{q_0+\delta}\big)^{\frac{q_0}{q_0+\delta}} 
			\notag
			\\
			&\quad+K\big(E\big(1+|x_{s}|^{\zeta q_0}+|y_{s}^n|^{\zeta q_0}\big)^{\frac{q_0+\delta}{\delta}}\big)^{\frac{\delta}{q_0+\delta}}\big(E|x_{s}-y_{s}^n|^{q_0+\delta}\big)^{\frac{q_0}{q_0+\delta}}
			\notag
			\\
			&\quad+K\big(E\big(1+|x_{s}^n|^{\zeta q_0}+|x_{\kappa_n(s)}^n|^{\zeta q_0}\big)^{\frac{q_0+\delta}{\delta}}\big)^{\frac{\delta}{q_0+\delta}}\big(E|x_{s}^n-x_{\kappa_n(s)}^n|^{q_0+\delta}\big)^{\frac{q_0}{q_0+\delta}}  \notag
			%\\
			%\leq & K n^{-\frac{q_0}{q_0+\delta}} \notag
		\end{align}
		which on utilizing Proposition \ref{prop:mb:sde}, Corollary \ref{cor:one_step_regularity_sde},   Remark \ref{rem:mb_auxilary}, Lemmas \ref{lem:estimate_first_term_MR}, \ref{lem:scm_mb},  and Corollary  \ref{cor:one_step_error},    completes the first term's estimate.
		Furthermore, for any $s\in[0,T]$, apply Remark \ref{rem:super_linear} and  Assumptions \ref{asum:poly_lip_jump} to \ref{asum:convergence} to obtain 
		\begin{align}
			E\!\!\int_Z&  \big|\gamma(s,x_{s},z)\!-\!\gamma(s, y_{s}^n,z)\big|^{q_0}\rho(dz)
			\!+\!E\!\int_Z \big|\gamma(s, x_{s}^n,z)\!-\!(\widehat\gamma)_{\tau m}^{\displaystyle{n}}\big(\kappa_n(s),x_{\kappa_n(s)}^{n}, z\big)\big|^{q_0} \rho(dz)\notag
			\\
			\leq
			&  E\!\int_Z\big|\gamma(s,x_{s},z)\!-\!\gamma(s, y_{s}^n,z)\big|^{q_0} \rho(dz) +KE\!\int_Z \big|\gamma(s,x_{s}^n,z)\!-\!\gamma(\kappa_n(s), x_{s}^n,z)\big|^{q_0} \rho(dz) \notag
			\\
			&\quad +KE\int_Z\big|\gamma(\kappa_n(s), x_{s}^{n},z)-\gamma(\kappa_n(s), x_{\kappa_n(s)}^{n},z)\big|^{q_0} \rho(dz)   \notag
			\\
			&\quad +KE\int_Z \big|\gamma(\kappa_n(s), x_{\kappa_n(s)}^{n},z)-(\widehat\gamma)_{\tau m}^{\displaystyle{n}}\big(\kappa_n(s),x_{\kappa_n(s)}^{n}, z\big)\big|^{q_0} \rho(dz)  \notag
			\\
			\leq &	KE\big((1+|x_{s}|^{\zeta}+|y_{s}^n|^{\zeta})|x_{s}-y_{s}^n|^{q_0}\big)  +K|s-\kappa_n(s)|^{\beta q_0}E\big(\Upsilon+|x_{s}^n|^{\zeta+ q_0}\big)  \notag
			\\
			&\quad+ KE\big((1+|x_{s}^n|^{\zeta}+|x_{\kappa_n(s)}^n|^{\zeta})|x_{s}^n-x_{\kappa_n(s)}^n|^{q_0}\big) + K n^{-\frac{q_0}{q_0+\delta}} \notag
			\\
			\leq & K\big(E\big(1+|x_{s}|^{\zeta}+|y_{s}^n|^{\zeta}\big)^{\frac{q_0+\delta}{\delta}}\big)^{\frac{\delta}{q_0+\delta}}\big(E|x_{s}-y_{s}^n|^{q_0+\delta}\big)^{\frac{q_0}{q_0+\delta}} +K n^{-\beta q_0}
			\notag
			\\
			&\quad+K\big(E\big(1+|x_{s}^n|^{\zeta}+|x_{\kappa_n(s)}^n|^{\zeta}\big)^{\frac{q_0+\delta}{\delta}}\big)^{\frac{\delta}{q_0+\delta}}\big(E|x_{s}^n-x_{\kappa_n(s)}^n|^{q_0+\delta}\big)^{\frac{q_0}{q_0+\delta}} + K n^{-\frac{q_0}{q_0+\delta}} \notag
			%\\
			%\leq & K n^{-\frac{q_0}{q_0+\delta}} \notag
		\end{align}
		which on using   Proposition \ref{prop:mb:sde},   Remark \ref{rem:mb_auxilary}, Lemmas \ref{lem:estimate_first_term_MR}, \ref{lem:scm_mb},  and Corollary  \ref{cor:one_step_error},    completes the last term's estimate.
		The estimate for the term related to  $\sigma$ is obtained following the same strategy as above.
	\end{proof}

	\begin{lemma} \label{lem:estimate_last_term_MR}
		Let	Assumptions \mbox{\normalfont  \ref{asum:ic}}, \mbox{\normalfont  \ref{asum:monotonocity}} and \mbox{\normalfont  \ref{asum:coercivity_p}} 
		to \mbox{\normalfont  \ref{asum:monotonocity_q0}} hold. 
		Then, 
		\begin{align*}
			&\sup_{ t\in[0,T]}E|y_t^n-x_t^{n}|^{q_0}
			\leq K n^{-\frac{q_0}{{q_0}+\delta}}
		\end{align*}
		for any   $q_0\geq 2$ and $\delta\in(0,1)$ such that $\max\big\{q_0\delta^{-1}(q_0+\delta)\zeta,(q_0+\delta)(\zeta+1)\big\}\leq q$, where  $K>0$ is a constant independent of $n\in\mathbb N$.
	\end{lemma}
	\begin{proof}
		Recall  \eqref{eq:scm}, \eqref{eq:auxiliary_equation} and use  It\^o's  lemma from \cite[Theorem 32]{Protter2005}  or \cite[Theorem 94]{Situ2006} to obtain the following
		\begin{align}
			|y_t^n-&x_t^{n}|^{q_0}\leq {q_0}\!\int_{0}^{t}\!|y_s^n-x_s^{n}|^{{q_0}-2}(y_s^n-x_s^{n})\big(\mu(\xi_n(s),x_{\kappa_n(s)})\!-\!(\widehat\mu)_{\tau m}^{\displaystyle{n}}\big(\xi_n(s),x_{\kappa_n(s)}^{n}\big)\big)ds  \notag
			\\
			&+q_0\int_{0}^{t}|y_s^n-x_s^{n}|^{{q_0}-2}(y_s^n-x_s^{n})\big(\sigma(s, x_s)- (\widehat\sigma)_{\tau m}^{\displaystyle{n}}(\kappa_n(s),x_{\kappa_n(s)}^{n})\big)dw_s  \notag
			\\
			& +\tfrac{q_0(q_0-1)}{2}\int_{0}^{t}|y_s^n-x_s^{n}|^{{q_0}-2}\big|\sigma(s, x_s)- (\widehat\sigma)_{\tau m}^{\displaystyle{n}}(\kappa_n(s),x_{\kappa_n(s)}^{n})\big|^2 ds \notag
			\\
			&+q_0\!\int_{0}^{t}\!\int_Z \!|y_s^n-x_s^{n}|^{{q_0}-2}(y_s^n-x_s^{n})\big(\gamma(s, x_s,z)\!-\!(\widehat\gamma)_{\tau m}^{\displaystyle{n}}\big(\kappa_n(s),x_{\kappa_n(s)}^{n}, z\big)\big)\tilde{n}_p(ds,dz)  \notag
			\\
			&+\int_{0}^{t}\int_Z \Big(\big|y_s^n-x_s^{n}+\gamma(s,x_s,z)-(\widehat\gamma)_{\tau m}^{\displaystyle{n}}\big(\kappa_n(s),x_{\kappa_n(s)}^{n}, z\big)\big|^{q_0}-|y_s^n-x_s^{n}|^{q_0}  \notag
			\\
			&\quad-q_0|y_s^n-x_s^{n}|^{{q_0}-2}(y_s^n-x_s^{n})\big(\gamma(s, x_s,z)-(\widehat\gamma)_{\tau m}^{\displaystyle{n}}\big(\kappa_n(s),x_{\kappa_n(s)}^{n}, z\big)\big)\Big){n}_p(ds,d z) \label{eq:error_y_n_x_n*}
		\end{align}
		almost surely for any $t\in[0,T]$, which on taking expectations yields 
		\begin{align} \label{eq:error_y_n_x_n}
			E | y_t^n -  x_t^{n}|^{q_0}   
			\leq &  q_0E\int_{0}^{t}|y_s^n-x_s^{n}|^{{q_0}-2}(y_s^n-x_s^{n})\big(\mu(\xi_n(s), x_{\kappa_n(s)})-(\widehat\mu)_{\tau m}^{\displaystyle{n}}\big(\xi_n(s),x_{\kappa_n(s)}^{n}\big)\big)ds \notag
			\\
			&
			+\tfrac{q_0(q_0-1)}{2} \int_{0}^{t}|y_s^n-x_s^{n}|^{{q_0}-2}\big|\sigma(s, x_s)- (\widehat\sigma)_{\tau m}^{\displaystyle{n}}(\kappa_n(s),x_{\kappa_n(s)}^{n})\big|^2ds \notag
			\\
			&+{q_0(q_0-1)}2^{q_0-4} \notag
            \\
            &\qquad \times  E\int_{0}^{t}\int_Z|y_s^n-x_s^{n}|^{{q_0}-2}\big|\gamma(s, x_s,z)-(\widehat\gamma)_{\tau m}^{\displaystyle{n}}\big(\kappa_n(s),x_{\kappa_n(s)}^{n}, z\big)\big|^2\rho(dz)ds \notag
			\\
			&+{q_0}(q_0-1)2^{q_0-4}E\int_{0}^{t}\int_Z\big|\gamma(s, x_s,z)-(\widehat\gamma)_{\tau m}^{\displaystyle{n}}\big(\kappa_n(s),x_{\kappa_n(s)}^{n}, z\big)\big|^{q_0}\rho(dz)ds  \notag
			\\
			=:&  \, \mathcal{T}_\mu +\mathcal{T}_\sigma+\mathcal{T}_\gamma  
		\end{align}
		where one utilizes the following first inequality  for the last term of  \eqref{eq:error_y_n_x_n*}
		\begin{align} 
			|a+b|^{q_0}-|a|^{q_0}-q_0|a|^{q_0-2}ab\leq &\, {q_0(q_0-1)}2^{q_0-4}\big(|a|^{q_0-2}|b|^2+|b|^{q_0}\big),   \, \mbox{ for  $q_0> 3 $} \label{eq:biprdt_mvt}
			\\
			\notag
			\\
			|a+b|^{q_0}-|a|^{q_0}-q_0|a|^{q_0-2}ab\leq& \, \tfrac{q_0(q_0-1)}{2}\big(|a|^{q_0-2}|b|^2+|b|^{q_0}\big) , \, \mbox{ for  $2<q_0\leq  3 $} \label{eq:biprdt_mvt1}
		\end{align}
		for any $a,b\in \mathbb R^{d}$, which are   due to Lemma \ref{lem:mvt}.
		Further, to estimate the first term  $\mathcal{T}_\mu$ of \eqref{eq:error_y_n_x_n}, apply Young's inequality and Lemma \ref{lem:estimate_tame_coeff_MR} to obtain
		\begin{align}
			\mathcal{T}_\mu: =& q_0E\int_{0}^{t}|y_s^n-x_s^{n}|^{{q_0}-2}(y_s^n-x_s^{n})\big(\mu(\xi_n(s), x_{\kappa_n(s)})-(\widehat\mu)_{\tau m}^{\displaystyle{n}}\big(\xi_n(s),x_{\kappa_n(s)}^{n}\big)\big)ds \notag
			\\
			\leq & q_0E\int_{0}^{t}|y_s^n-x_s^{n}|^{{q_0}-2}(y_s^n-x_s^{n})\big(\mu(\xi_n(s),x_{\kappa_n(s)})-\mu( \xi_n(s), y_{s}^n)\big)ds \notag
			\\
			&+ q_0E\int_{0}^{t}|y_s^n-x_s^{n}|^{{q_0}-2}(y_s^n-x_s^{n})\big(\mu( \xi_n(s), y_{s}^n)-\mu(\xi_n(s),  x_{s}^n)\big)ds \notag
			\\
			&+q_0E\int_{0}^{t}|y_s^n-x_s^{n}|^{{q_0}-2}(y_s^n-x_s^{n})\big(\mu(\xi_n(s), x_{s}^n)-(\widehat\mu)_{\tau m}^{\displaystyle{n}}\big(\xi_n(s),x_{\kappa_n(s)}^{n}\big)\big)ds \notag
			\\
			\leq & K E \int_0^t|y_s^n-x_s^{n}|^{q_0}ds+KE\int_{0}^{t}\big|\mu(\xi_n(s),x_{\kappa_n(s)})-\mu( \xi_n(s), y_{s}^n)\big|^{q_0}ds  \notag
			\\
			&+q_0E\int_{0}^{t}|y_s^n-x_s^{n}|^{q_0-2}(y_s^n-x_s^{n})\big(\mu(\xi_n(s),y_s^n)-\mu(\xi_n(s),x_s^{n})\big) ds  \notag
			\\
			& +KE\int_{0}^{t}\big|\mu(\xi_n(s), x_{s}^{n})-(\widehat\mu)_{\tau m}^{\displaystyle{n}}\big(\xi_n(s),x_{\kappa_n(s)}^{n}\big)\big|^{q_0}ds  \notag
			\\
			\leq & K n^{-\frac{q_0}{q_0+\delta}} +K E \int_0^t|y_s^n-x_s^{n}|^{q_0}ds  \notag
			\\
			&+q_0E\int_{0}^{t}|y_s^n-x_s^{n}|^{q_0-2}(y_s^n-x_s^{n})\big(\mu(\xi_n(s),y_s^n)-\mu(\xi_n(s),x_s^{n})\big) ds \label{eq:Tmu}
		\end{align}
		for any $t\in[0,T]$.
		Moreover, to handle the second term $\mathcal{T}_\sigma$ of  \eqref{eq:error_y_n_x_n}, one requires the following inequality for any $a,b\in \mathbb R^{d_1\times m_1}$, $d_1, m_1\geq 1$, $r\geq 2$ and $\lambda >1$,
		\begin{align}
			|a+b|^{r} = \Big|\frac{1}{\lambda}\lambda a+\frac{\lambda-1}{\lambda}\frac{\lambda}{\lambda-1}b\Big|^{r} &\leq \frac{1}{\lambda}|\lambda a|^{r}+\frac{\lambda-1}{\lambda}\Big|\frac{\lambda}{\lambda-1}b\Big|^{r} \notag
			\\
			&=  \lambda^{r-1}|a|^{r} +\big(\frac{\lambda}{\lambda-1}\big)^{r-1}|b|^{r}  \label{eq:conseq:jensen:ineq}
		\end{align}
		which is the consequence of    Jensen's inequality applied to the convex function  $|x|^{r}$, $x\in \mathbb R^{d_1\times m_1}$.
		Then, by Equation \eqref{eq:conseq:jensen:ineq} for $r=2$ together with Young's inequality and Lemma \ref{lem:estimate_tame_coeff_MR}, one obtains 
		\begin{align} \label{eq:Tsigma}
			\mathcal{T}_\sigma: =& \tfrac{q_0(q_0-1)}{2}E\int_{0}^{t}|y_s^n-x_s^{n}|^{{q_0}-2}\big|\sigma(s, x_s)- (\widehat\sigma)_{\tau m}^{\displaystyle{n}}(\kappa_n(s),x_{\kappa_n(s)}^{n})\big|^2ds \notag
			\\
			\leq &
			\tfrac{q_0(q_0-1)\lambda}{2(\lambda-1)}E\!\!\int_{0}^{t}\!\!|y_s^n\!-\!x_s^{n}|^{{q_0}-2}\big|\sigma(s, x_s)\!-\!\sigma(s, y_s^n)\!+\!\sigma(s, x_s^n)\!- \!(\widehat\sigma)_{\tau m}^{\displaystyle{n}}(\kappa_n(s),x_{\kappa_n(s)}^{n})\big|^2 \!ds \notag
			\\
			&
			\qquad\qquad\quad+\tfrac{q_0(q_0-1)\lambda}{2}E\int_{0}^{t}|y_s^n-x_s^{n}|^{{q_0}-2}\big|\sigma(s, y_s^n)-\sigma(s, x_s^n)\big|^2ds 
			\notag
			\\
			\leq &
			KE\int_{0}^{t}|y_s^n-x_s^{n}|^{{q_0}}ds
			+KE\int_{0}^{t}\big|\sigma(s, x_s)-\sigma(s, y_s^n)\big|^{q_0} ds \notag
			\\
			&
			\qquad\qquad\quad+KE\int_{0}^{t}\big|\sigma(s, x_s^n)- (\widehat\sigma)_{\tau m}^{\displaystyle{n}}(\kappa_n(s),x_{\kappa_n(s)}^{n})\big|^{q_0}ds \notag
			\\
			&
			\qquad\qquad\quad+\tfrac{q_0(q_0-1)\lambda}{2}\int_{0}^{t}|y_s^n-x_s^{n}|^{{q_0}-2}\big|\sigma(s, y_s^n)-\sigma(s, x_s^n)\big|^2ds \notag
			\\
			\leq & K n^{-\min\big\{\frac{q_0}{q_0+\delta},\,\alpha q_0\big\}} +K E \int_0^t|y_s^n-x_s^{n}|^{q_0}ds  \notag
			\\
			& \qquad\qquad\quad+\tfrac{q_0(q_0-1)\lambda}{2}\int_{0}^{t}|y_s^n-x_s^{n}|^{{q_0}-2}\big|\sigma(s, y_s^n)-\sigma(s, x_s^n)\big|^2ds
		\end{align}
		for any $t\in [0,T]$. 
		Further, for the analysis of the third term $\mathcal{T}_\gamma$ of  \eqref{eq:error_y_n_x_n}, utilize  \eqref{eq:conseq:jensen:ineq} as follows 
		\begin{align} 
			\mathcal{T}_\gamma: =& 
			{q_0(q_0-1)}2^{q_0-4}E\!\!\int_{0}^{t}\!\!\int_Z|y_s^n \!-\!x_s^{n}|^{{q_0}-2}\big|\gamma(s, x_s,z)\!-\!(\widehat\gamma)_{\tau m}^{\displaystyle{n}}\big(\kappa_n(s),x_{\kappa_n(s)}^{n}, z\big)\big|^2\rho(dz)ds \notag
			\\
			&+{q_0}(q_0-1)2^{q_0-4}E\int_{0}^{t}\int_Z\big|\gamma(s, x_s,z)-(\widehat\gamma)_{\tau m}^{\displaystyle{n}}\big(\kappa_n(s),x_{\kappa_n(s)}^{n}, z\big)\big|^{q_0}\rho(dz)ds \notag
			\\
			\leq &
			\tfrac{q_0(q_0-1)2^{q_0-4}\lambda}{(\lambda-1)}E\int_{0}^{t}\int_Z|y_s^n-x_s^{n}|^{{q_0}-2}\big|\gamma(s, x_s,z)-\gamma(s, y_s^n,z)  \notag
			\\
			&\qquad\qquad\qquad\qquad\qquad\qquad+\gamma(s, x_s^n,z)-(\widehat\gamma)_{\tau m}^{\displaystyle{n}}\big(\kappa_n(s),x_{\kappa_n(s)}^{n}, z\big)\big|^2 \rho(dz)ds \notag
			\\
			& +
			{q_0(q_0-1)2^{q_0-4}\lambda}E\int_{0}^{t}\int_Z|y_s^n-x_s^{n}|^{{q_0}-2}\big|\gamma(s, y_s^n,z) -\gamma(s, x_s^n,z)\big|^2 \rho(dz)ds \notag
			\\
			&+
			\tfrac{q_0(q_0-1)2^{q_0-4}\lambda^{q_0-1}}{(\lambda-1)^{q_0-1}}E\int_{0}^{t}\int_Z \big|\gamma(s, x_s,z)-\gamma(s, y_s^n,z)  \notag
			\\
			&\qquad\qquad\qquad\qquad\qquad\qquad+\gamma(s, x_s^n,z)-(\widehat\gamma)_{\tau m}^{\displaystyle{n}}\big(\kappa_n(s),x_{\kappa_n(s)}^{n}, z\big)\big|^{q_0} \rho(dz)ds \notag
			\\
			& +
			{q_0(q_0-1)2^{q_0-4}\lambda^{q_0-1}}E\int_{0}^{t}\int_Z\big|\gamma(s, y_s^n,z) -\gamma(s, x_s^n,z)\big|^{q_0} \rho(dz)ds \notag
		\end{align}
		for any $t\in[0,T]$, which on using  Young's inequality and Lemma  \ref{lem:estimate_tame_coeff_MR} yields 
		\begin{align}\label{eq:Tgamma}
			\mathcal{T}_\gamma 
			\leq & K E\int_{0}^{t}|y_s^n-x_s^{n}|^{q_0}ds +KE\int_0^t\Big(\int_Z \big|\gamma(s, x_s,z)-\gamma(s, y_s^n,z)\big|^2\rho(dz)\Big)^{{q_0}/2}ds \notag
			\\
			&+KE\int_0^t\Big(\int_Z \big|\gamma(s, x_s^n,z)-(\widehat\gamma)_{\tau m}^{\displaystyle{n}}\big(\kappa_n(s),x_{\kappa_n(s)}^{n}, z\big)\big|^2\rho(dz)\Big)^{{q_0}/2}ds \notag
			\\
			& +
			{q_0(q_0-1)2^{q_0-4}\lambda}E\int_{0}^{t}\int_Z|y_s^n-x_s^{n}|^{{q_0}-2}\big|\gamma(s, y_s^n,z) -\gamma(s, x_s^n,z)\big|^2 \rho(dz)ds \notag
			\\
			&+
			KE\int_{0}^{t}\int_Z \big|\gamma(s, x_s,z)-\gamma(s, y_s^n,z)\big|^{q_0} \rho(dz)ds  \notag
			\\
			&+KE\int_{0}^{t}\int_Z \big|\gamma(s, x_s^n,z)-(\widehat\gamma)_{\tau m}^{\displaystyle{n}}\big(\kappa_n(s),x_{\kappa_n(s)}^{n}, z\big)\big|^{q_0} \rho(dz)ds \notag
			\\
			& +
			{q_0(q_0-1)2^{q_0-4}\lambda^{q_0-1}}E\int_{0}^{t}\int_Z\big|\gamma(s, y_s^n,z) -\gamma(s, x_s^n,z)\big|^{q_0} \rho(dz)ds 
			\notag
			\\
			\leq & K n^{-\min\big\{\frac{q_0}{q_0+\delta},\,\beta q_0\big\}} +K E \int_0^t|y_s^n-x_s^{n}|^{q_0}ds  \notag
			\\
			& +
			{q_0(q_0-1)2^{q_0-4}\lambda}E\int_{0}^{t}\int_Z|y_s^n-x_s^{n}|^{{q_0}-2}\big|\gamma(s, y_s^n,z) -\gamma(s, x_s^n,z)\big|^2 \rho(dz)ds \notag
			\\
			& +
			{q_0(q_0-1)2^{q_0-4}\lambda^{q_0-1}}E\int_{0}^{t}\int_Z\big|\gamma(s, y_s^n,z) -\gamma(s, x_s^n,z)\big|^{q_0} \rho(dz)ds.
		\end{align}
		Finally, substitute  \eqref{eq:Tmu}, \eqref{eq:Tsigma} and \eqref{eq:Tgamma} in \eqref{eq:error_y_n_x_n}, and use Assumption \ref{asum:monotonocity_q0} to obtain
		\begin{align*}
			&\sup_{ r\in[0,t]}E|y_r^n-x_r^{n}|^{q_0}
			\leq K n^{-\min\big\{\frac{q_0}{q_0+\delta},\,\alpha q_0,\,\beta q_0\big\}} + K \int_{0}^{t}\sup_{ r\in[0,s]}E|y_r^n-x_r^{n}|^{q_0}ds 
		\end{align*}
		which, upon applying Gr\"onwall's inequality, completes the proof for the case \( q_0 > 3 \).
		
		For the case \( 2 < q_0 \leq 3 \),  use \eqref{eq:biprdt_mvt1} 
		in place of \eqref{eq:biprdt_mvt} within Equation~\eqref{eq:error_y_n_x_n} 
		and follow the same strategy as above.
		
		For \( q_0 = 2 \),   apply the identity
		\(
		|a+b|^2 - |a|^2 - 2ab = |b|^2
		\)
		instead of 
		\eqref{eq:biprdt_mvt} 
		within Equation~\eqref{eq:error_y_n_x_n}, and then proceed by the same arguments as above.
		%For the case \(  q_0< 2  \), one uses H\"older's inequality.
	\end{proof}
	%To obtain the desired convergence rate, one needs to introduce an intermediate system governed by \eqref{eq:auxiliary_equation}, which serves as a bridge between the original process \eqref{eq:sde} and its numerical approximation \eqref{eq:scm}.
	\begin{proof}[Proof of Theorem \ref{thm:main_result}]
		Add and subtract Equation \eqref{eq:auxiliary_equation} for any $t\in[0,T]$ as follows
		\begin{align}
			E|x_t-x_t^{n}|^{q_0}\leq K E|x_t-y_t^{n}|^{q_0}+  K E|y_t^n-x_t^{n}|^{q_0} \notag
		\end{align}
		which on using Lemmas \ref{lem:estimate_first_term_MR} and \ref{lem:estimate_last_term_MR} completes the proof.
	\end{proof}

	\section{Numerical Experiment} \label{sec:numerics}
	This section supports our theoretical findings with numerical experiments on new variants of the \textit{double-well} SDEs, which naturally fit into the proposed framework.

	Specifically, a one-dimensional L\'evy-driven {double-well} SDE  is considered
	\begin{align} \label{eq:DWD}
		x_t=x_0 + \int_0^t 
		(\beta(s)&	\,x_s- \hat\beta x_s^3)\,ds +   \int_0^t \hat\sigma\,\sqrt s\,(1- x_s^2)\, dw_s \notag
		\\
		&+\int_0^t\int_{\mathbb R} \hat\gamma\, \sqrt s \, x_s(1+x_s^2)^{1/{p}}z\,\tilde{n}_p(ds, dz)
	\end{align}
	almost surely for any $t\in [0, 1]$ with initial value $x_0=2$. Here, $\beta(s)$ is  the sawtooth function given by $ 
	\beta(s) = 2 \left( s - \left\lfloor s + \frac{1}{2} \right\rfloor \right), \, s \in [0,1]
	$ and the parameters are  
	\( p = 648\), 
	\(\hat\beta = 0.5\), 
	\(\hat\sigma = 0.001\), 
	and \(\hat\gamma = 0.02\).
	See Remark \ref{rem:parameters} for how these choices fit into this setting.
	%$\beta(s) = 1 + 0.5 \left(\lfloor 5s \rfloor \bmod 2\right)$. 
	
	The jump sizes follow a normal distribution with mean zero and variance one. The reference solution is computed via the randomized tamed Euler method \eqref{eq:scm}, employing taming coefficients defined in \eqref{eq:tame_coeffs}, with a step size of \(2^{-18}\).
	The results presented in the following table and figures are based on 1900 independent simulation runs.
	\begin{table}[h!]
		\centering
		\renewcommand{\arraystretch}{1.3}
		\begin{tabular}{|c|c|c|c|c|}
			\hline
			Step-size $(\Delta t)$ 
			& $\mathcal{L}^1$-error 
			& $\mathcal{L}^2$-error 
			& $\mathcal{L}^3$-error 
			& $\mathcal{L}^4$-error \\
			\hline
			$2^{-8}$  & 0.0505168099 & 0.0505596037 & 0.0506024896 & 0.0506455842 \\
			$2^{-9}$  & 0.0354740285 & 0.0355301495 & 0.0355862603 & 0.0356426459 \\
			$2^{-10}$ & 0.0249509386 & 0.0250259560 & 0.0251006265 & 0.0251756723 \\
			$2^{-11}$ & 0.0175644721 & 0.0176660328 & 0.0177662608 & 0.0178670418 \\
			$2^{-12}$ & 0.0123686827 & 0.0125078057 & 0.0126428726 & 0.0127789515 \\
			$2^{-13}$ & 0.0087087569 & 0.0088994042 & 0.0090799058 & 0.0092628381 \\
			$2^{-14}$ & 0.0061346808 & 0.0063903604 & 0.0066260755 & 0.0068683170 \\
			$2^{-15}$ & 0.0043477010 & 0.0046676499 & 0.0049636411 & 0.0052750296 \\
			$2^{-16}$ & 0.0031492921 & 0.0035125050 & 0.0038670096 & 0.0042486419 \\
			$2^{-17}$ & 0.0023850943 & 0.0027677091 & 0.0031747856 & 0.0036142790 \\
			\hline
		\end{tabular}
		\vspace{0.4cm}
		\caption{\textit{$\mathcal{L}^p$-error of  randomized tamed Euler scheme  \eqref{eq:scm} for SDE  \eqref{eq:DWD}. }}
	\end{table}

	\begin{center}
		\begin{minipage}{.4\textwidth}
			\centering
			\includegraphics[width=1\linewidth, height=0.7\linewidth]{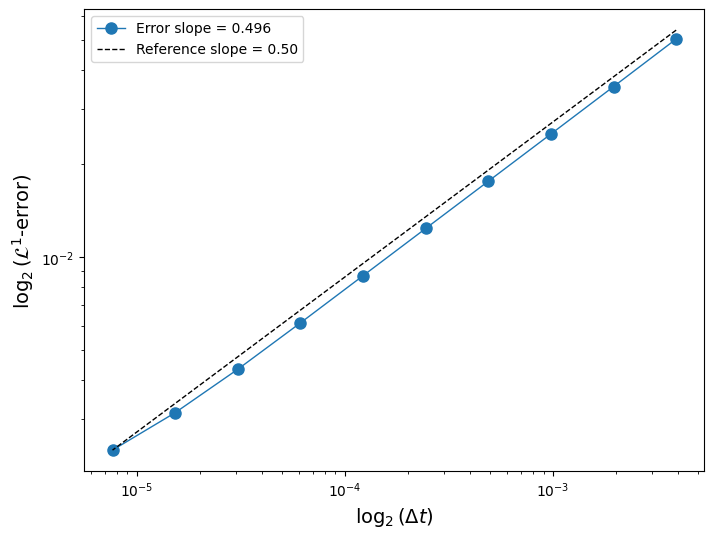}
			%	\vspace{-0.7cm}
			%\captionof*{figure}{\quad\, p=1}
			%\label{fig:woutCN:random}
		\end{minipage} 
		\begin{minipage}{0.4\textwidth}
			\centering
			\includegraphics[width=1\linewidth, height=0.65\linewidth]{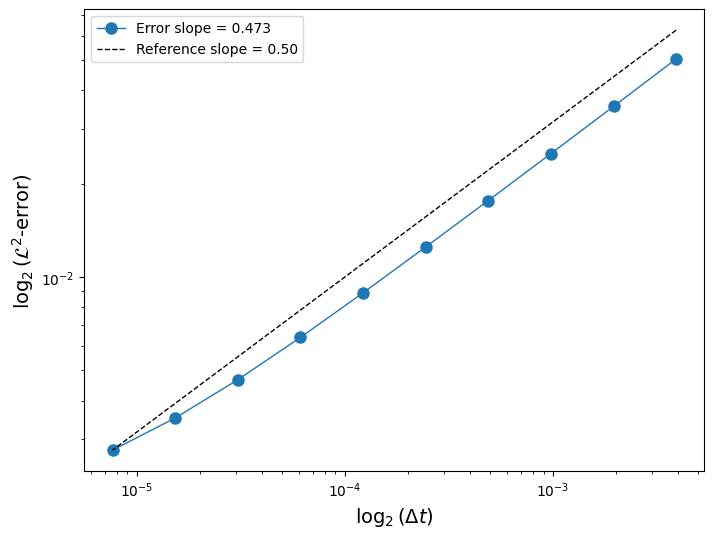}
			%	\vspace{-0.7cm}
			%\captionof*{figure}{\quad\, p=2}
			%\label{fig:wCN::random}
		\end{minipage}
	\end{center}
	\begin{center}
		\begin{minipage}{.4\textwidth}
			\centering
			\includegraphics[width=1\linewidth, height=0.65\linewidth]{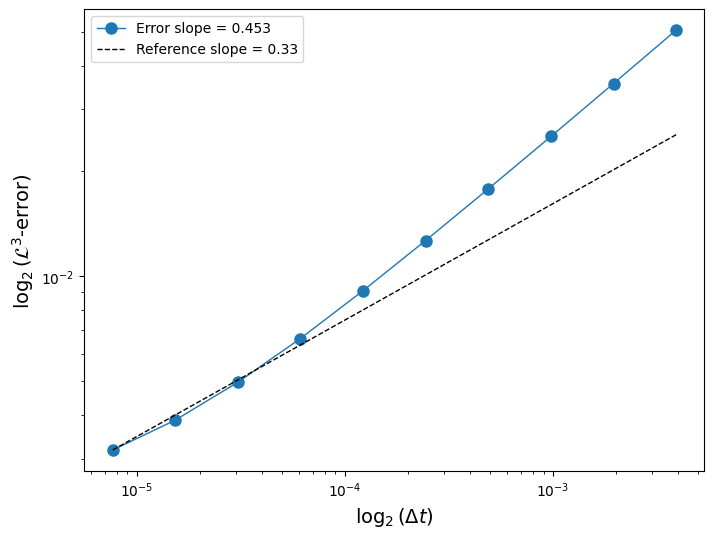}
			%\vspace{-0.7cm}
			%	\captionof*{figure}{\quad\, p=3}
			%\label{fig:woutCN:random}
		\end{minipage}
		\begin{minipage}{0.4\textwidth}
			\centering
			\includegraphics[width=1\linewidth, height=0.65\linewidth]{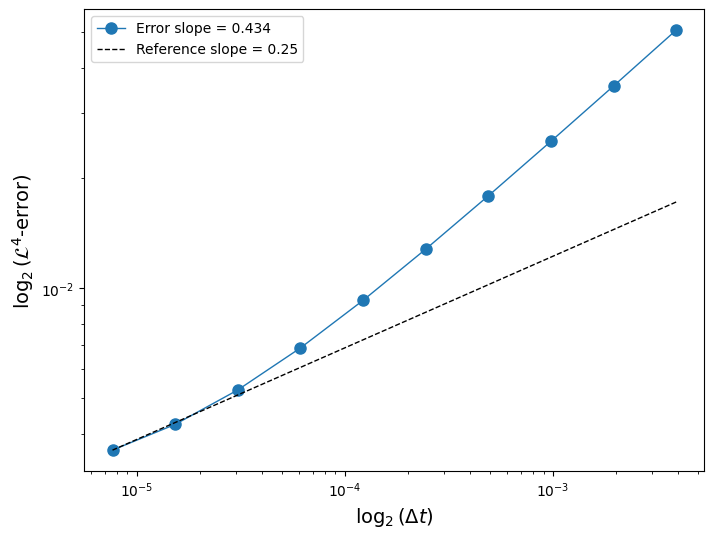}
			%	\vspace{-0.7cm}
			%\captionof*{figure}{\quad\, p=4}
			%\label{fig:wCN::random}
		\end{minipage}
		\captionof{figure}{\textit{$\mathcal{L}^p$-error of  randomized tamed Euler scheme  \eqref{eq:scm} for SDE  \eqref{eq:DWD}.} }
	\end{center}

    \medskip
In the cases $p=1$ and $2$, the numerical experiments closely align with the 
theoretically optimal convergence rates, approximately $0.5$, 
thus providing further validation of Theorem~\ref{thm:main_result}. 
For $p=3$ and $4$, the numerical results indicate even faster convergence rates 
than the theoretical approximate values of $0.33$ and $0.25$ established by 
Theorem~\ref{thm:main_result}. This behavior is intuitive and expected in many scenarios, 
as the theorem offers general upper bounds not necessarily tight for $p > 2$. 
Nevertheless, the theoretical estimates remain valid and are not contradicted 
by these observations.

	\section*{Acknowledgements}
	This research was partially 
	funded by Centro de Modelamiento Matem\'atico (CMM) FB210005 BASAL fund for centers of excellence and Fondecyt Grant 1242001, both from ANID-Chile. 
	%------------------------------------------------

		\appendix
	\section{Verifications of Assumptions} \label{sec:app}
	This section provides a rigorous verification that the taming coefficients introduced in Section~\ref{sec:example} satisfy Assumptions~\ref{asum:coercivity_scm}, \ref{asum:tame} and \ref{asum:convergence}, which are essential for the analysis of the randomized tamed Euler scheme \eqref{eq:scm} applied  to the  SDE \eqref{eq:sde}. Additionally, it is shown that the coefficients of the variant of the double-well dynamics SDE \eqref{eq:DWD} satisfy Assumptions~\ref{asum:ic} to \ref{asum:coercivity_p} as well as \ref{asum:at_0} to \ref{asum:holder_time_diffusion_jump} and \ref{asum:monotonocity_q0}.
	
	%\hspace{0.4cm}
	Firstly, all required assumptions on the taming coefficients employed in the scheme \eqref{eq:scm} are duly verified. Importantly, the analysis does not hinge on the specific form of the coefficients presented in \eqref{eq:DWD}, thereby ensuring broader applicability.

	\medskip
	\textit{Verification of Assumption \ref{asum:coercivity_scm}.}
	Recall  \eqref{eq:tame_coeffs} and use Assumption \ref{asum:coercivity_p} to get
	\begin{align*}%\label{eq:coercivity:p*} 
		&	2|x|^{q-2}x(\widehat{\mu})_{\tau m}^{\displaystyle{n}}(s,x)
		+(q-1)|x|^{q-2}\big|(\widehat\sigma)_{\tau m}^{\displaystyle{n}}(t,x)\big|^2 
		+2(q-1)\int_{Z}\big| (\widehat{\gamma})_{tam}^{\displaystyle{n}}(t,x,z) \big|^2 \notag
		\\
		&\qquad  \qquad\qquad
		\times\int_{0}^1 (1-\theta)  \big |x+\theta|(\widehat{\gamma})_{tam}^{\displaystyle{n}}(t,x,z)| \big|^{q-2} d\theta \rho(dz)  \notag
		\\
		\leq& \frac{2|x|^{q-2}x\mu(s,x)}{1+n^{-1/2}|x|^{3\zeta/2}}
		+\frac{(q-1)|x|^{q-2}|\sigma(t,x)|^2}{(1+n^{-1/2}|x|^{3\zeta/2})^2} +2(q-1)\int_{Z}\frac{|\gamma(t,x,z)|^2}{(1+n^{-1/2}|x|^{3\zeta/2})^2}  \notag
		\\
		&\qquad  \qquad\qquad
		\times\int_{0}^1 (1-\theta)   \Big|x+\frac{\theta|\gamma(t,x,z)|}{1+n^{-1/2}|x|^{3\zeta/2}}\Big| ^{q-2} d\theta \rho(dz) \notag
		\\
		&\leq \frac{1}{1+n^{-1/2}|x|^{3\zeta/2}}  \Big(2|x|^{{q}-2}x\mu(s,x)
		+(q-1)|x|^{{q}-2}|\sigma(t,x)|^2+2(q-1)\int_{Z}|{\gamma(t,x,z)} |^2 \notag
		\\
		&\qquad  \qquad \qquad \times \int_{0}^1 (1-\theta)\big|x+\theta|\gamma(t,x,z)|\big| ^{q-2}d\theta\rho(dz)  \Big)
		\leq K(1+|x|)^{q}
	\end{align*}
	for any \( s \in [0,T] \) and  \( x \in \mathbb{R}^d \). %where the last step also follows $\big|x+\theta(\widehat{\gamma})_{tam}^{\displaystyle{n}}(s,x,z) \big|\leq \big|x+\theta{\gamma}(s,x,z) \big| $, which holds, for instance, when   ${\gamma}(x,z)=|x|(1+|x|^2)^{p}$ with $p\leq q$, for any $x\in \mathbb R^d$ and $z\in Z$.
	
	\medskip
	\textit{Verification of Assumption \ref{asum:tame}.}  Equation \eqref{eq:tame_coeffs} and  Remark  \ref{rem:super_linear} yield as follows
	\begin{align} 
		|(\widehat{\mu})_{\tau m}^{\displaystyle{n}}(s,x)|=&\frac{|\mu(s,x)|}{1+n^{-1/2}|x|^{3\zeta/2}}\leq \frac{K(\Upsilon+|x|^{\zeta+1})}{(1+n^{-1/3}|x|^{\zeta})^{3/2}} \leq \frac{K(\Upsilon+|x|^{\zeta+1})}{1+n^{-1/3}|x|^{\zeta}} \notag
		\\
		\leq & Kn^{1/3}(1+|x|^{\zeta})\frac{\Upsilon+|x|}{1+|x|^{\zeta}}\leq Kn^{1/3}(\Upsilon+|x|) \notag
	\end{align}
	for any \( s \in [0,T] \) and  \( x \in \mathbb{R}^d \).
	Also, one  obtains by using  \eqref{eq:tame_coeffs} and  Remark  \ref{rem:super_linear} again
	\begin{align} 
		|(\widehat\sigma)_{\tau m}^{\displaystyle{n}}(s,x)|=&\frac{|\sigma(s,x)|}{1+n^{-1/2}|x|^{3\zeta/2}}\leq \frac{K(\Upsilon+|x|^{\zeta/2+1})}{(\Upsilon+n^{-1/6}|x|^{\zeta/2})^{3}} \leq \frac{K(\Upsilon+|x|^{\zeta/2+1})}{1+n^{-1/6}|x|^{\zeta/2}} \notag
		\\
		\leq & Kn^{1/6}(1+|x|^{\zeta/2})\frac{\Upsilon+|x|}{1+|x|^{\zeta/2}}\leq Kn^{1/6}(\Upsilon+|x|) \notag
	\end{align}
	for any \( s \in [0,T] \) and  \( x \in \mathbb{R}^d \).
	Once again employ  \eqref{eq:tame_coeffs} and  Remark  \ref{rem:super_linear} to have
	\begin{align*}
		\int_{Z}|(\widehat\gamma)_{\tau m}^{\displaystyle{n}}(s,x,z)|^{p_0} \rho(dz)=&\int_{Z}\frac{|\gamma(s,x,z)|^{p_0}}{(1+n^{-1/2}|x|^{3\zeta/2})^{p_0}}\rho(dz)
		\leq \frac{K(\Upsilon+|x|^{\zeta+{p_0}})}{(1+n^{-1/3}|x|^{\zeta})^{3{p_0}/2}}  \notag
		\\
		\leq& \frac{K(\Upsilon+|x|^{\zeta+{p_0}})}{1+n^{-1/3}|x|^{\zeta}} \leq  Kn^{1/3}(1+|x|^{\zeta})\frac{\Upsilon+|x|^{p_0}}{1+|x|^{\zeta}}
		\\
		& \qquad\qquad\leq Kn^{1/3}(\Upsilon+|x|^{p_0}) \notag
	\end{align*}
	for any \( s \in [0,T] \),  \( x \in \mathbb{R}^d \) and $p_0\geq 2$.
	
	\medskip
	\textit{Verification of Assumption \ref{asum:convergence}.} Use  \eqref{eq:tame_coeffs},  Remark  \ref{rem:super_linear} and Young's inequality to obtain
	\begin{align}
		E\big|\mu\big(\xi_n(s),x_{\kappa_n(s)}^{n}\big) &- (\widehat\mu)_{\tau m}^{\displaystyle{n}}\big(\xi_n(s),x_{\kappa_n(s)}^{n}\big)\big|^{p_0}\!\!
		=E\Big|\mu\big(\xi_n(s),x_{\kappa_n(s)}^{n}\big)\Big(1-\frac{1}{1+n^{-1/2}|x_{\kappa_{n}(s)}^{n}|^{3\zeta/2}}\Big)\Big|^{p_0} \notag
		\\
		& \quad \leq n^{-p_0/2} \sup_{ s\in[0,T]}E\big(\big|\mu\big(s,x_{\kappa_n(s)}^{n}\big)\big|^{p_0}\big|x_{\kappa_n(s)}^{n}\big|^{3\zeta p_0/2}\big) \notag
		\\
		& \quad \leq K n^{-{p_0}/{2}}\!\!\sup_{ s\in[0,T]}E\big(\Upsilon^{p_0}|x_{\kappa_n(s)}^{n}\big|^{3\zeta p_0/2}+\big|x_{\kappa_n(s)}^{n}\big|^{5\zeta p_0/2+p_0}\big) \notag
		\\
		& \quad \leq K n^{-{p_0}/{2}}\!\!\sup_{ s\in[0,T]}E\big(\Upsilon^{p_0(\zeta+1)}+|x_{\kappa_n(s)}^{n}\big|^{3(\zeta+1) p_0/2}+\big|x_{\kappa_n(s)}^{n}\big|^{5\zeta p_0/2+p_0}\big) \notag
		\\
		& \quad \leq K n^{-\frac{p_0}{p_0+\delta}} \notag %\label{eq:tame_drift_conv}
	\end{align}
	for any $s\in[0,T]$, ${\delta \in(0,1)}$ and $2\leq p_0\leq 2q/(5\zeta+3)$, where the last step follows  Lemma \ref{lem:scm_mb} 
    %along with the condition $(5\zeta+3)p_0/2\leq q$.
    %\max\{p_0(\zeta+1),(5\zeta+3)p_0/2\}\leq q$.
	Similarly, Assumption~\ref{asum:convergence}  holds for the  term related to the diffusion coefficient.  
	Furthermore,  use  \eqref{eq:tame_coeffs}, Remark   \ref{rem:super_linear} and Young's inequality to get 
	\begin{align*}
		& \Big(\int_{Z}E\big| \gamma(\kappa_n(s),x_{\kappa_n(s)}^{n},z)-(\widehat\gamma)_{\tau m}^{\displaystyle{n}}\big(\kappa_n(s),x_{\kappa_n(s)}^{n}, z\big)\big|^{\bar q}\rho(dz)\Big)^{p_0/\bar q}
		\\
		& \qquad \quad\leq n^{-p_0/2} \Big(\int_{Z}E\big(\big|\gamma(\kappa_n(s),x_{\kappa_n(s)}^{n},z)\big|^{\bar q}\big|x_{\kappa_n(s)}^{i,N,n}\big|^{3\zeta \bar q/2}\big)\rho(dz)\Big)^{p_0/\bar q}
		\\
		& \qquad \quad\leq K n^{-{p_0}/{2}}\Big(E\big(\Upsilon|x_{\kappa_n(s)}^{n}\big|^{3\zeta \bar q/2}+\big|x_{\kappa_n(s)}^{n}\big|^{\zeta+\bar q +3\zeta \bar q/2}\big)\Big)^{p_0/\bar q}
		\\
		& \qquad \quad\leq  K n^{-{p_0}/{2}}\Big(E\big(\Upsilon^{4}+|x_{\kappa_n(s)}^{n}\big|^{2\zeta \bar q}\big)\Big)^{p_0/\bar q}+K n^{-{p_0}/{2}}
		\leq K n^{-\frac{p_0}{p_0+\delta}}, \,\,\,\mbox{ $\bar q=2, p_0$}
	\end{align*}
	for any $s\in[0,T]$, ${\delta \in(0,1)}$ and $p_0\geq 2$, where the last step follows  Lemma \ref{lem:scm_mb}  with the  previously  considered condition on $p_0$. %$\max\{q_0(\zeta+1),(5\zeta+3)q_0/2\leq q\}$.
	Consequently, Theorem~\ref{thm:main_result} remains valid under the illustrative choice of taming coefficients in \eqref{eq:tame_coeffs}.
   % , provided that the additional condition $(5\zeta+3)p_0/2\leq q$ is imposed.
	
	\smallskip
	Furthermore, we verify Assumptions~\ref{asum:coercivity_p}, \ref{asum:poly_lip_drift}, \ref{asum:poly_lip_jump}, \ref{asum:holder_time_diffusion_jump}, and \ref{asum:monotonocity_q0}. Subsequently, the remaining assumptions related to the coefficients of the SDE~\eqref{eq:DWD} are either trivially satisfied or follow directly from the verified conditions.
	A general form of the coefficients of the SDE~\eqref{eq:DWD} are defined by
	\begin{align} \label{eq:coeff:GL}
		\mu(s,x) = \beta(s) x - \hat{\beta} x^3, \quad
		\sigma(s,x) = \hat{\sigma} \sqrt{s} (1 - x^2), \quad
		\gamma(s,x,z) = \hat{\gamma} \sqrt{s}\, x (1 + x^2)^{1/p} z , \mbox{   $p\geq q$}
	\end{align}
	for any \( s \in [0,1] \) and \( x, z \in \mathbb{R} \),
	where  $ 
	\beta(s) = 2 \left( s - \left\lfloor s + \frac{1}{2} \right\rfloor \right)
	$ and \(\hat{\beta}, \hat{\sigma}, \hat{\gamma}\) are positive parameters. 
	%One may consult Remark \ref{rem:parameters} below for a discussion on how these parameters fit within our framework.

	\medskip
	\textit{Verification of Assumption \ref{asum:coercivity_p}.}  
	Use \eqref{eq:coeff:GL} and Young's inequality together with    the condition $p\geq q> 2$   and the setting
	\(
	m_{\bar p} := \int_{\mathbb R} |z|^{\bar p} \, \rho(dz),
	\) to obtain
	\begin{align}
		2|x|^{q-2} &x\,\mu(s,x)
		+ (q-1) |x|^{q-2} \bigl|\sigma(t,x)\bigr|^2 \notag 
		\\
		&+ 2(q-1) \int_{\mathbb R} \bigl|\gamma(t,x,z)\bigr|^2
		\int_0^1 (1-\theta) \big|x+\theta|\gamma(t,x,z)|\big|^{q-2} \, d\theta\, \rho(dz) \notag 
		\\
		= \; & 2|x|^{q-2} x\bigl(\beta(s)x - \hat\beta x^3\bigr)
		+ \hat\sigma^2 t (q-1) |x|^{q-2} (1-x^2)^2 \notag 
		\\
		&+ 2(q-1) \int_{\mathbb R} 
		\big| \hat\gamma\, \sqrt t\, z\, x(1+x^2)^{1/p} \big|^2 
		\int_0^1 (1-\theta) \big|x+\theta \hat\gamma \, \sqrt t\, |z x|(1+x^2)^{1/p} \big|^{q-2} d\theta\, \rho(dz) \notag 
		\\
		\leq\; &   K|x|^q - 2\hat\beta |x|^{q+2} 
		+ \hat\sigma^2 (q-1) |x|^{q-2}(1-2x^2+x^4) \notag
		\\
		& + (q-1) \int_{\mathbb R}
		\big(\hat\gamma |z x| (1+x^2)^{1/p}\big)^2 
		\big( |x| + \hat\gamma |z x| (1+x^2)^{1/p}  \big)^{q-2}
		\rho(dz) \notag
		\\
		\leq\; &  K+K |x|^q - \bigl(2\hat\beta - \hat\sigma^2 (q-1)\bigr) |x|^{q+2}  \notag
		\\
		&+ 2^{q-1} (q-1) \int_{\mathbb R}
		\left(\hat\gamma^2 |z|^2 |x|^q (1+x^2)^{2/p} 
		+ \hat\gamma^q |z x|^q (1+x^2)^{q/p} \right) \rho(dz) \notag 
		%\\
		%\leq\; & K+K |x|^q 
		%	- \bigl(2\hat\beta - \hat\sigma^2 (q-1)\bigr) |x|^{q+2}
		%	+ 2^{q-1} (q-1) \big(\hat\gamma^2 m_2+\hat\gamma^q m_q\big) |x|^{q+2}	\leq K \bigl(1 + |x|^q\bigr) \notag 
		\\
		\leq\; & K+K |x|^q 
		-  \Big( 2\hat\beta - \hat\sigma^2 (q-1) 
		- 2^{q-1} (q-1) \big(\hat\gamma^2 m_2+\hat\gamma^q m_q\big) \Big)|x|^{q+2}
		\leq K \bigl(1 + |x|^q\bigr)\notag
	\end{align}
	for any $s,t\in[0,1]$ and \( x \in \mathbb R \),
	{since} 
	\begin{align} \label{eq:for_gamhat1}
		\hat\sigma^2 (q-1) + 2^{q-1} (q-1) \big(\hat\gamma^2 m_2+\hat\gamma^q m_q\big) \leq 2\hat\beta.
	\end{align}
	
	%	\[
	%	\hat\gamma 
	%	\leq \exp\!\left(\frac{-113.3 -5.09 -336}{164}\right)
	%	= e^{-2.77}
	%	\approx 0.063.
	%	\]

	\medskip
	\textit{Verification of Assumption \ref{asum:poly_lip_drift}.} 
	Recall  \eqref{eq:coeff:GL} and apply Young's inequality to  obtain 
	\begin{align} 
		|\mu(s,x)-\mu(s,y)|
		& = \big|\beta(s) x-\hat\beta x^3- \beta(s)y+ \hat\beta y^3\big| \notag
		\\
		&\leq  K\big|x-y|+K|x-y||x^2+xy+y^2|  \notag
		\\
		&\leq K|x-y|\big(1+|x|^2+|y|^2) \label{eq:ploy_lip_drift}
	\end{align}
	for any $s\in[0,1]$  and $x,y\in\mathbb R$.

	\medskip
	\textit{Verification of Assumption \ref{asum:poly_lip_jump}.} 
	Recall  \eqref{eq:coeff:GL} and then by direct computation 
	\begin{align}
		\partial_x \gamma(s,x,z)
		= \hat\gamma \sqrt{s} \,(1+x^2)^{1/p}\, z
		+ \frac{2\hat\gamma \sqrt{s}}{p} x^2 (1+x^2)^{\frac{1-p}{p}}\, z
		\notag
	\end{align}
	which allows application of the mean value theorem  to yield with the condition $p\geq q\geq 4$,
	\begin{align}
		\gamma(s,x,z) - \gamma(s,y,z)
		=& \hat\gamma \sqrt{s}\, z\, (x - y)
		\int_0^1 \!\Bigl(
		(1+(y+\theta(x-y))^2)^{1/p} \notag
		\\
		&+ \tfrac{2}{p} (y+\theta(x-y))^2 (1+(y+\theta(x-y))^2)^{\frac{1-p}{p}}
		\Bigr)\, d\theta 
		\notag
		\\
		\leq& \tfrac{3}{2}\hat\gamma \, |z|\, |x - y|
		\int_0^1 \!
		\bigl(1+(y+\theta(x-y))^2\bigr)^{1/p} \, d\theta 
		\label{eq:gam_dif}
	\end{align}
	for any \( s \in [0,1] \) and \( x,y,z\in\mathbb R \). Further, 	for any \( s \in [0,1] \) and \( x,y\in\mathbb R \), use H\"older's inequality together with Equation \eqref{eq:gam_dif} and the condition $p\geq q\geq p_0$,
	\begin{align}
		\int_{\mathbb R}|\gamma(s,x,z) - \gamma(s,y,z)|^{p_0} \rho(dz)
		\leq& K |x - y|^{p_0}
		\int_0^1\bigl(1+(y+\theta(x-y))^2\bigr)^{p_0/p} \, d\theta  \notag
		\\
		\leq& K |x - y|^{p_0}
		\!(1+|x|^2+|y|^2).   \label{eq:ploy_lip_gam}
	\end{align}

	\medskip
	\textit{Verification of Assumption \ref{asum:holder_time_diffusion_jump}.} Note that for any fixed \(x,z \in \mathbb{R}\), the diffusion and jump coefficients of the SDE \eqref{eq:DWD} satisfy \(\tfrac{1}{2}\)--H\"older continuity in time. Therefore, Assumption \ref{asum:holder_time_diffusion_jump} is
	fulfilled.

	\medskip
	\textit{Verification of Assumption \ref{asum:monotonocity_q0}.}  
	By recalling \eqref{eq:coeff:GL}, and employing    H\"older's inequality and the  condition $p\geq q\geq p_0\geq 2$ in \eqref{eq:gam_dif}   along with the mean value theorem  and the setting
	\(
	m_{\bar p} = \int_{\mathbb R} |z|^{\bar p} \, \rho(dz)\), one obtains
	\begin{align*}
		p_0&|x-y|^{p_0-2}(x-y)\big(\mu(s,x)-\mu(s,y)\big)
		+\tfrac{p_0(p_0-1)\lambda}{2}|x-y|^{{p_0}-2}\big|\sigma(t, x)-\sigma(t, y)\big|^2
		\\
		& \qquad+
		{p_0(p_0-1)\bigl(2^{p_0-4}\mathbf{1}_{\{p_0>3,\, p_0= 2\}}+\frac{1}{2}\mathbf{1}_{\{2<p_0\leq 3\}}\bigr)} \notag
		\\
		&\qquad \quad \times \int_Z\Big(\lambda|x-y|^{{p_0}-2}\big|\gamma(t, x,z) -\gamma(t, y,z)\big|^2 +\lambda^{p_0-1}\big|\gamma(t, x,z) -\gamma(t, y,z)\big|^{p_0}\Big) \rho(dz)\notag
		\\
		&\quad\leq p_0|x-y|^{p_0-2}(x-y)\bigl(\beta(s)x - \hat\beta x^3 - \beta(s)y + \hat\beta y^3\bigr)
		+ \tfrac{p_0(p_0-1)\lambda \hat\sigma^2 t}{2} |x-y|^{p_0-2}|x^2 - y^2|^2 
		\notag
		\\
		&\qquad + \tfrac{9}{4}p_0\bigl(p_0-1)(2^{p_0-4}+\tfrac{1}{2}\bigr)\lambda \hat\gamma^2 \, m_2 |x-y|^{p_0} 
		\int_0^1 \!\bigl(1 + (y+\theta(x-y))^2\bigr)^{2/p} d\theta \notag
		\\
		&\qquad + (\tfrac{3}{2})^{p_0}p_0\bigl(p_0-1)(2^{p_0-4}+\tfrac{1}{2}\bigr)\lambda^{p_0-1} \hat\gamma^{p_0} \, m_{p_0} |x-y|^{p_0} 
		\int_0^1 \!\bigl(1 + (y+\theta(x-y))^2\bigr)^{p_0/p} d\theta
		\\
		&\quad\leq p_0|x-y|^{p_0-2}(x-y) \Bigl(\beta(s)(x-y) - 3\hat\beta (x-y)
		\int_0^1 \!\bigl(y+\theta(x-y)\bigr)^2\, d\theta\Bigr)
		\notag\\
		&\qquad + 2p_0(p_0-1)\lambda \hat\sigma^2  (x-y)^{p_0} 
		\int_0^1 \!\bigl(y+\theta(x-y)\bigr)^2\, d\theta
		\\
		&\qquad + \tfrac{9}{4}p_0\bigl(p_0-1)(2^{p_0-4}+\tfrac{1}{2}\bigr)\lambda \hat\gamma^2 \, m_2 |x-y|^{p_0} 
		\int_0^1 \!\bigl(1 + (y+\theta(x-y))^2\bigr) d\theta \notag
		\\
		&\qquad + (\tfrac{3}{2})^{p_0}p_0\bigl(p_0-1)(2^{p_0-4}+\tfrac{1}{2}\bigr)\lambda^{p_0-1} \hat\gamma^{p_0} \, m_{p_0} |x-y|^{p_0} 
		\int_0^1 \!\bigl(1 + (y+\theta(x-y))^2\bigr) d\theta
		\\
		&\quad\leq K |x-y|^{p_0} 
		- \Bigl(3p_0\hat\beta - 2p_0(p_0-1)\lambda \hat\sigma^2 - p_0\bigl(p_0-1)(2^{p_0-4}+\tfrac{1}{2}\bigr)
		\\
		&\qquad\qquad \qquad\times\bigl(\tfrac{9}{4}\lambda \hat\gamma^2 \, m_2  +(\tfrac{3}{2})^{p_0}\lambda^{p_0-1} \hat\gamma^{p_0} \, m_{p_0}\bigr) \Bigr)
		|x-y|^{p_0}  
		\int_0^1 \!\bigl(y+\theta(x-y)\bigr)^2\, d\theta
		\notag
		\\
		&\quad\leq K |x-y|^{p_0}% =:\mathcal C_s^n|x-y|^{q_0},\notag \,\,\mbox{ with \,$ E^{\varphi}(\mathcal C_s^n)\leq K+q_0$}
	\end{align*}
	for any $s,t\in[0,1]$, \( x,y\in\mathbb R\) and \(\lambda>1\), since
	\begin{align}
		2(p_0-1)\lambda \hat\sigma^2 
		+\bigl(p_0-1)(2^{p_0-4}+\tfrac{1}{2}\bigr)\bigl(\tfrac{9}{4}\lambda \hat\gamma^2 \, m_2  +(\tfrac{3}{2})^{p_0}\lambda^{p_0-1} \hat\gamma^{p_0} \, m_{p_0}\bigr)\leq 	3\hat\beta . \label{eq:for_gamhat}
	\end{align}
	%	\begin{align*}
		%		&2(q_0 - 1)\lambda \hat\sigma^2 + (q_0 - 1) \left( 2^{q_0 - 3} + \frac{1}{2} \right) \left( \frac{9}{4} \lambda \hat\gamma^2 m_2 + \left( \frac{3}{2} \right)^{q_0} \lambda^{q_0 - 1} \hat\gamma^{q_0} m_{q_0} \right) \leq 3 \hat\beta, \\
		%		&\text{with } q_0=2, \lambda=1.001, \hat\sigma=0.001, m_2=1, m_4=3, \hat\beta=0.5, \\
		%		&\implies 2 \times 1 \times 1.001 \times (0.001)^2 + 1.75 \times \left( \frac{9}{4} \times 1.001 \hat\gamma^2 + 2.25 \times 1.001 \hat\gamma^2 \right) \leq 3 \times 0.5, \\
		%		&\implies \boxed{2.002 \times 10^{-6}} + \boxed{7.883 \hat\gamma^2} \leq \boxed{1.5}, \\
		%		&\implies 7.883 \hat\gamma^2 \leq 1.5 - 2.002 \times 10^{-6} \approx 1.499998, \\
		%		&\implies \hat\gamma^2 \leq \boxed{0.1904}, \quad \hat\gamma \leq \boxed{0.436}.
		%	\end{align*}

	%\medskip

	\medskip
	Now, the parameters for the numerical simulations of the randomized tamed Euler scheme \eqref{eq:scm}, based on the taming in \eqref{eq:tame_coeffs}, for the SDE \eqref{eq:DWD}, can be chosen as follows.
	\begin{remark} \label{rem:parameters}
		From Equations~\eqref{eq:ploy_lip_drift} and~\eqref{eq:ploy_lip_gam}, it follows that $\zeta = 2$. 
		With the choice $\delta = 0.05$ in Theorem~\ref{thm:main_result}, 
		the condition
		\(
		\max\bigl\{ 
		p_0\,\delta^{-1}(p_0+\delta)\,\zeta,\;
		(p_0+\delta)(\zeta+1),(5\zeta+3)p_0/2
		\bigr\}
		\leq q, 
		\text{ for } p_0 \in \{1,2,3,4\},
		\) leads to the explicit bound $q \geq 648$. 
		Accordingly, it is suitable to take $p \geq 648$ in the jump coefficients of the SDE~\eqref{eq:DWD}. 
		Moreover, with $\lambda = 1.001$ as introduced in Assumption~\ref{asum:monotonocity_q0}, together with the parameters 
		$\hat\beta = 0.5$ and $\hat\sigma = 0.001$ in~\eqref{eq:DWD}, 
		it follows from Equations~\eqref{eq:for_gamhat1} and~\eqref{eq:for_gamhat}, which involve $m_{\bar p}$, the $\bar p$-th moment of the standard normal distribution, that
		\(
		\hat\gamma \;\lessapprox\; 0.02.
		\)
	\end{remark}
	
\end{document}